\documentclass[a4paper,11pt, twoside]{amsart}

\usepackage[utf8]{inputenc}
\usepackage[inner=2.5cm,outer=2.5cm,top=3cm,headsep=1cm, footskip=1cm,bottom=3cm]{geometry}
\usepackage[english]{babel}
\usepackage{amsfonts,latexsym,rawfonts,amsmath,amssymb,amsthm,amscd}
\usepackage{fancyhdr}
\usepackage{enumerate}
\usepackage{stackrel}

\usepackage{hyperref}
\usepackage{mdwlist}
\usepackage{array} 
\usepackage[all]{xy}

\hypersetup{
    colorlinks,%
    citecolor=black,%
    filecolor=black,%
    linkcolor=black,%
    urlcolor=black
}

\makeatletter
\let\uppercasenonmath\@gobble% disables title uppercase
\makeatother

\usepackage{titlesec}
\titleformat{\section}{\bfseries\center}{\thesection.}{0.4em}{}
\titleformat{\subsection}{\vspace{.08cm}\bfseries}{\thesubsection.}{0.4em}{}

\newtheorem{prop}{Proposition}[section]
\newtheorem{teo}[prop]{Theorem}
\newtheorem{lem}[prop]{Lemma}
\newtheorem{cor}[prop]{Corollary}

\theoremstyle{definition}
\newtheorem{nada}[prop]{}
\newtheorem{defi}[prop]{Definition}
\newtheorem{example}[prop]{Example}
\newtheorem{rmk}[prop]{Remark}

\theoremstyle{theorem}
\newcommand{\eps}{\varepsilon}

\newcommand{\inte}{\text{{\scriptsize$\int_0^1$}}}

\def\Ho{\mathrm{Ho}}

\def\Dec{\mathrm{Dec}}

\def\Ker{\mathrm{Ker }}

\newcommand{\pb}{\ar@{}[dr]|{\mbox{\LARGE{$\lrcorner$}}}}

\newcommand{\hto}{\rightsquigarrow}
\newcommand{\xra}[1]{\xrightarrow{#1}}

\newcommand{\dga}[2]{\mathsf{DGA}^{#1}_{#2}}
\newcommand{\Fdga}[2]{\mathbf{F}\mathsf{DGA}^{#1}_{#2}}
\newcommand{\FFdga}[2]{\mathbf{F}^2\mathsf{DGA}^{#1}_{#2}}
\newcommand{\Sdga}[2]{\mathbf{S}\mathsf{min}^{#1}_{#2}}

\newcommand{\GCc}{\Gamma\Cc}
\newcommand{\GCch}{\Gamma\Cc^h}
\newcommand{\pGCch}{\pi^h\Gamma\Cc}
\newcommand{\DFAlg}[1]{\Gamma\mathbf{A}^{#1}}
\newcommand{\DFCx}{\Gamma\mathbf{C}}

\newcommand{\MHD}{\mathbf{MHD}}
\newcommand{\MHSA}{\mathbf{MH}{\mathsf{DGA}}_{\normalfont{min}}}
\newcommand{\MHC}{\mathbf{MHC}}
\newcommand{\MHS}{\mathsf{MHS}}

\newcommand{\wt}{\widetilde}
\newcommand{\lra}{\longrightarrow}

\newcommand{\ov}[1]{\overline{#1}}

\newcommand{\CC}{\mathbb{C}}

\newcommand{\QQ}{\mathbb{Q}}
\newcommand{\RR}{\mathbb{R}}

\newcommand{\ZZ}{\mathbb{Z}}

\newcommand{\Cc}{\mathcal{C}}
\newcommand{\Dd}{\mathcal{D}}
\newcommand{\Ee}{\mathcal{E}}
\newcommand{\Ff}{\mathcal{F}}
\newcommand{\Hh}{\mathcal{H}}

\newcommand{\Ll}{\mathcal{L}}
\newcommand{\Mm}{\mathcal{M}}

\newcommand{\Pp}{\mathcal{P}}
\newcommand{\Qq}{\mathcal{Q}}
\newcommand{\Ss}{\mathcal{S}}
\newcommand{\Ww}{\mathcal{W}}

\newcommand{\kk}{\mathbf{k}}
\newcommand{\Sch}[2]{\mathbf{Sch}^{#1}(#2)}

\title[\sc \normalsize cofibrant models of diagrams: mhs in rational homotopy]{\Large{Cofibrant models of diagrams:\\ mixed Hodge structures in rational homotopy
}}
\author[\sc \normalsize joana cirici] {\large Joana Cirici}
\address{
Fachbereich Mathematik und Informatik\\
Freie Universit\"{a}t Berlin\\  Arnimallee 3\\ 
14195 Berlin}
\email{jcirici@math.fu-berlin.de}

\thanks{Financial support by
the Dahlem Research School - Marie Curie Action through PCOFUND-GA-2010-267228.
}


\date{\today}

\begin{document}
\maketitle
\begin{abstract}
We study the homotopy theory of a certain type of diagram categories
whose vertices are in variable categories with a functorial path, leading to a 
good calculation of the homotopy category in terms of cofibrant objects.
The theory is applied to the category of mixed Hodge diagrams of differential graded algebras.
Using Sullivan's minimal models, we prove a multiplicative version of Beilinson's Theorem on mixed Hodge complexes. As a consequence, we obtain
functoriality for the mixed Hodge structures on the rational homotopy type of complex algebraic varieties.
In this context, the mixed Hodge structures on homotopy groups
obtained by Morgan's theory follow from the derived functor of the indecomposables of mixed Hodge diagrams.
\end{abstract}

\section{Introduction}
Since the development of Sullivan’s rational homotopy theory \cite{Su}, minimal models
have found significant applications
of both topological and geometric origin,
one of the first and most striking being the Formality Theorem 
of Deligne-Griffiths-Morgan-Sullivan \cite{DGMS} for compact K\"{a}hler manifolds.
Also using Sullivan's minimal models and based on Deligne's mixed Hodge theory \cite{DeHII},
Morgan \cite{Mo} proved the existence of functorial mixed Hodge structures on the rational homotopy groups of
smooth complex algebraic varieties. 
His results were independently extended to the singular case by Hain \cite{Ha} and Navarro \cite{Na}.
Both works depend on the initial constructions of Morgan.
The main objects under consideration in Morgan's theory are mixed Hodge diagrams of differential graded algebras, a multiplicative
analogue of the mixed Hodge complexes of Deligne \cite{DeHIII} involving differential graded algebras over $\QQ$ and $\CC$ respectively
and encoding the weight and Hodge filtrations up to filtered quasi-isomorphisms.

Bousfield and Gugenheim \cite{BG} reformulated Sullivan's rational homotopy theory
of differential graded algebras
in the context of Quillen model categories \cite{Q1}.
Following this line, it would be desirable to establish an analogous formulation
for mixed Hodge diagrams of differential graded algebras.
In this sense, though sufficient for its original purposes, Morgan's theory of mixed Hodge diagrams is incomplete, 
since it provides the existence of certain minimal models,
but these are not shown to be cofibrant or minimal in any abstract categorical framework.
Moreover, Morgan allows morphisms between diagrams to be homotopy commutative and does not claim any composition law.
As a consequence, his results fall out of the realm of categories.
This is one aspect that motivates the present work.

Driven by motivic and Deligne cohomology,
Beilinson \cite{Be} studied the homotopy category of mixed Hodge complexes and proved that it is equivalent 
to the derived category of mixed Hodge structures, allowing an interpretation of
Deligne's cohomology in terms of extensions of mixed Hodge structures.
Our objective in this paper is to prove a multiplicative analogue of Beilinson's equivalence, allowing to understand the results of
Deligne, Beilinson, Morgan, Hain and Navarro within a common homotopical framework.

The axioms for Quillen's model categories are very powerful and they provide, not only a precise description of
the maps in the homotopy category, but also other higher homotopical structures. As a counterpart,
there exist interesting categories from the homotopical 
point of view, which do not satisfy all the axioms.
This is the case of diagram categories involving filtrations, where more specific techniques have to be introduced.

In this paper we develop an abstract homotopy theory for certain diagram categories with vertices in variable categories.
We follow the homotopical approach of Cartan-Eilenberg categories introduced by 
Guill\'{e}n-Navarro-Pascual-Roig
\cite{GNPR}, a weaker framework than the one provided by Quillen
model structures, but sufficient to study homotopy categories and to extend
the classical theory of derived additive functors, to non-additive settings.
An important observation is that in this setting,
one can consider minimal models as a
particular case of cofibrant ones,
parallel to Sullivan's theory.
The theory is then applied to the category of mixed Hodge diagrams of differential graded algebras:
we describe morphisms in the homotopy
category of mixed Hodge diagrams in terms of certain homotopy classes of morphisms between Sullivan minimal algebras carrying mixed Hodge structures.
Together with Navarro's functorial construction of mixed Hodge diagrams \cite{Na}, this gives functoriality for the mixed Hodge structures
on the rational homotopy type of complex algebraic varieties.
In this context, the mixed Hodge structures on the rational homotopy groups
obtained by Morgan's theory follow from the derived functor of the indecomposables of mixed Hodge diagrams. Our approach can be applied to broader settings,
such as the study of complex analytic spaces with a class of compactifications, for which the Hodge and weight filtrations
can be defined (see \cite{GN} and \cite{To}), but do not satisfy the strong properties of mixed Hodge theory.

The present paper is a natural extension of previous works written by the author jointly with F. Guill\'{e}n:
here we generalize the results of \cite{CG2} to the multiplicative setting, using
the construction of minimal models for mixed Hodge diagrams appearing in \cite{CG1}.
\\

We describe the content of the different sections. Let $I$ be a small category and 
let $\Cc:I\to \mathsf{Cat}$ be a functor with values in the category of categories. 
Denote by $\int_I\Cc$ the Grothendieck construction of the functor $\Cc$ (see e.g. \cite{Th}). 
The category of diagrams $\GCc$ associated with $\Cc$ is defined as the category of sections of the projection $\int_I\Cc\to I$.
Objects of $\GCc$ are given by families of objects $\{A_i\in\Cc(i)\}$ for every $i\in I$, together with morphisms
$\{\varphi_u:\Cc(u)(A_i)\to A_j\}$ for every map $u:i\to j$ of $I$. Morphisms of $\GCc$ are families 
of level-wise morphisms in $\Cc(i)$ making the corresponding diagrams commute.
It is already an important question in abstract homotopy theory to know whether given compatible homotopical structures on the categories $\Cc(i)$,
there exists an induced homotopical structure on $\GCc$ with level-wise weak equivalences.
For categories of diagrams $\Cc^I$ associated with a constant functor there are partial answers 
in terms of Quillen model structures: if $\Cc$ is cofibrantly generated,
or $I$ has a Reedy structure, then the category $\Cc^I$ inherits a level-wise model structure
(see for example \cite{Hov}, Theorem 5.2.5).
It is also well known that if $\Cc$ is a Brown category of (co)fibrant objects \cite{Br}, then $\Cc^I$ inherits a Brown category structure,
with weak equivalences and (co)fibrations defined level-wise.
Here we study this question in the context of Cartan-Eilenberg categories
and provide a positive answer for a certain type of diagram categories
whose vertices are endowed with a functorial path.

In Section 2 we introduce P-categories and study their homotopy theory. A \textit{P-category} is given by a category
$\Cc$ with a functorial path $P:\Cc\to\Cc$ and two classes of 
morphisms $\Ff$ and $\Ww$ of \textit{fibrations} and \textit{weak equivalences} satisfying certain
axioms close to those of Brown categories of cofibrant objects, together 
with a homotopy lifting property with respect to trivial fibrations (see Definition $\ref{BCE_cat}$).
The functorial path defines a notion of homotopy between morphisms of $\Cc$.
We introduce a notion of cofibrant object and show that
if $C$ is cofibrant, then every weak equivalence $w:A\to B$
induces a bijection
$w_*:[C,A]\to [C,B]$
between homotopy classes of morphisms.
In Theorem $\ref{Pcat_es_CE}$ we show that if every object has a cofibrant model, 
then $\Cc$ admits a Cartan-Eilenberg structure with level-wise weak equivalences.
In particular, the inclusion induces an equivalence of categories
$\pi\Cc_{cof}\stackrel{\sim}{\lra}\Ho(\Cc).$

In Section 3 we develop the basic examples of P-categories: these are the category of topological spaces and the category of
differential graded algebras. We also provide a criterion of structure transfer
and apply it to two fundamental examples appearing in mixed Hodge theory:
 the categories of filtered and bifiltered differential graded algebras.

In Section 4 we study the homotopy theory of diagram categories.
It is quite immediate, that if the vertices of a diagram category $\GCc$ are endowed with compatible P-category structures,
then the diagram category inherits a level-wise P-category
structure. However, the characterization and existence of cofibrant models of diagrams is not straightforward,
and requires a rectification theory of homotopy commutative morphisms.
We focus our study to diagrams indexed by a finite directed category of binary degree (see $\ref{indexcat}$).
This includes the diagrams of zig-zag type appearing in mixed Hodge theory. 
We call \textit{ho-morphisms} those maps between diagrams that commute up to fixed homotopies. In general, ho-morphisms cannot be composed. 
However, the level-wise functorial path of $\GCc$ defines a notion of homotopy between ho-morphisms.
Consider the full subcategory of level-wise cofibrant diagrams. Its objects,
together with the homotopy classes of ho-morphisms define a category $\pGCch_{cof}$ (see Theorem $\ref{equiv_homorf_hoequiv}$).
In Theorem $\ref{diagramesCE0}$ we show that if the categories $\Cc_i$ have enough cofibrant models, then the
category of diagrams $\GCc$ admits a Cartan-Eilenberg structure with the same weak equivalences. In particular
there is an equivalence of categories
$\pGCch_{cof}\stackrel{\sim}{\lra}\Ho(\GCc).$
We also consider a relative situation
suitable to study mixed Hodge diagrams (see Theorem $\ref{diagramesCE}$).

Section 5 is devoted to the applications to multiplicative mixed Hodge theory.
A \textit{mixed Hodge diagram of differential graded algebras}
is given by a filtered dga $(A_\QQ,W)$ defined over $\QQ$, a bifiltered
dga $(A_\CC,W,F)$ defined over $\CC$, together with a string of
filtered quasi-isomorphisms
$(A_\QQ,W)\otimes\CC\longleftrightarrow (A_\CC,W)$ over $\CC$. In addition, the filtrations should satisfy certain axioms,
making the triple $(H(A_\QQ),\Dec W, F)$ into a graded mixed Hodge structure.
Denote by $\MHD$ the
category of mixed Hodge diagrams, and by
 $\MHSA$ the full subcategory of \textit{minimal mixed Hodge dga's}: these are Sullivan minimal dga's $A$ over $\QQ$ with filtrations $W$ on $A$ and $F$
on $A\otimes_\QQ\CC$ such that for each $n\geq 0$ the triple $(A^n,\Dec W,F)$ is a mixed Hodge structure.
In Theorem $\ref{MHDCE}$ we prove a multiplicative version of Beilinsons's Theorem 
on mixed Hodge complexes (Theorem 3.4 of \cite{Be}, see also Theorem 4.11 of \cite{CG2}), by showing
that the category $\MHD$ admits a Cartan-Eilenberg structure with cofibrant minimal models in $\MHSA$, where the weak equivalences are
level-wise quasi-isomorphisms compatible with filtrations.
As a consequence, we obtain an
equivalence of categories
$\pi^h\MHSA\stackrel{\sim}{\lra} \Ho\left(\MHD\right).$
As a corollary we show, using Navarro's functorial construction of mixed Hodge diagrams,
that the rational homotopy type of every complex algebraic variety is endowed
functorial mixed Hodge structures (see Corollary $\ref{mhstipohomotopia}$). This 
solves the lack of functoriality of Morgan's theory. Furthermore, in
Theorem $\ref{leftderivedMHD}$ we derive the functor of indecomposables of 1-connected mixed Hodge diagrams.
This leads to a more precise and alternative construction of functorial mixed Hodge
structures on
the rational homotopy groups of simply connected
varieties (see Corollary $\ref{mhd_mhs}$).

\section{P-categories and Cofibrant Models}
In the present section we introduce P-categories. These are categories with a functorial path and
two distinguished classes of morphisms, called
\textit{fibrations} and \textit{weak equivalences}, satisfying a list of axioms similar to those of Brown categories of fibrant objects \cite{Br}.
We introduce a notion of cofibrant object in terms of a lifting property with respect to trivial fibrations
and prove that every P-category with enough cofibrant models is a Cartan-Eilenberg category with
the same weak equivalences. As a consequence,
the homotopy category is equivalent to the quotient category of cofibrant objects
modulo homotopy.
Basic examples of P-categories are the category of topological spaces or the category of commutative differential graded 
algebras over a field of characteristic zero.

\subsection{Categories with a functorial path}
We recall the main results on categories
with a functorial path (see also Section I.4 of \cite{KP}).
\begin{defi}\label{funct_path}
A \textit{functorial path} on a category $\Cc$ is a functor $P:\Cc\to \Cc$
together with natural transformations
$$\xymatrix{1\ar[r]^-{\iota}&P \ar@<.5ex>[r]^-{\delta^0} \ar@<-.5ex>[r]_-{\delta^1}&1}
\text{ such that } \delta^0\iota=\delta^1\iota=1.
$$

\end{defi}

\begin{defi}\label{homotopia_assoc_path}
%  Let $(\Cc,P)$ be a category with a functorial path and let 
Let $f,g:A\to B$ be two morphisms of $\Cc$. A \textit{homotopy from $f$ to $g$} is
a morphism $h:A\to P(B)$ of $\Cc$ such that $\delta^0_Bh=f$ and $\delta^1_Bh=g$.
We use the notation $h:f\simeq g$.
\end{defi}
\begin{lem}[\cite{KP}, Lemma I.2.3]\label{reflexiva}
The homotopy relation defined by a functorial path is reflexive and compatible with the composition.
\end{lem}

\begin{nada}\label{cilindre_top}
We shall consider extra structure on the path.
The notion dual to the functorial path is that of a functorial cylinder. A basic example of such construction
is the product $X\times I$ of a topological space $X$ with the unit interval $I=[0,1]$.
In this case one has the following operators:
\begin{enumerate}[1.]
 \item Symmetry. The automorphism of $I$ defined by $t\mapsto 1-t$ makes the homotopy relation into a symmetric relation.
 \item Interchange. There is an automorphism of $I^2:=I\times I$ defined by $(t,s)\mapsto (s,t)$.
 \item Product. There is a map $I^2\to I$ given by $(t,s)\mapsto ts$.
 \item Diagonal map. There is a map $I\to I^2$ given by $t\mapsto (t,t)$.
\end{enumerate}
We next axiomatize these transformations in their dual version.
\end{nada}

Given a functorial path $P$, we denote $P^0=1$, $P^1=P$, $P^2=PP$, $\cdots$.
For all $0\leq s\leq n$ we have natural transformations
$$\xymatrix{P^n\ar[r]^{\iota^{n,s}}&P^{n+1} \ar@<1ex>[r]^{(\delta^0)^{n,s}} \ar@<-1ex>[r]_{(\delta^1)^{n,s}}&P^n},\quad
\left\{\begin{array}{l}
 \iota^{n,s}:=P^s(\iota_{P^{n-s}})\\
 (\delta^k)^{n,s}:=P^s(\delta^k_{P^{n-s}})
\end{array}\right..
$$
%The triple $(\iota, \delta^0, \delta^1)$ makes $P$
%into a cubical object in the category $\Fun(\Cc,\Cc)$.

\begin{defi}[\cite{KP}, Def. I.4.1]
A \textit{symmetry} of $P$ is a natural automorphism $\tau:P\to P$ such that
$\tau_A\tau_A=1_{P(A)}$, $\tau_A\iota_A=\iota_A$ and $\delta^k_A\tau_A=\delta^{1-k}_A$ for $k=0,1$.
\end{defi}

\begin{lem}[\cite{KP}, Prop I.4.5]\label{simetrica}
The homotopy relation defined by a functorial path with a symmetry is a symmetric relation.
\end{lem}
\begin{defi}[\cite{KP}, Def. I.4.6]\label{coproducte}
A \textit{coproduct} of $P$ is a natural transformation $c:P\to P^2$ such that:
\begin{enumerate}[(a)]
 \item The triple $(P,\delta^1,c)$ is a comonad, i.e. for $A\in \Cc$ one has commutative diagrams
$$\xymatrix{
P(A)\ar[r]^{c_A}\ar[d]_{c_A}&P^2(A)\ar[d]^{c_{P(A)}}\\
P^2(A)\ar[r]^{P(c_A)}&P^3(A)&,
}\xymatrix{
P(A)&P^2(A)\ar[l]_{\delta^1_{P(A)}}\ar[r]^{P(\delta^1_{A})}&P(A)\\
&P(A)\ar[ul]^{1_{P(A)}}\ar[ur]_{1_{P(A)}}\ar[u]^{c_A}&.
}
$$
\item For all $A\in\Cc$ the following diagrams commute.
$$
\xymatrix{
A\ar[d]_{\iota_A}\ar[r]^{\iota_A}&P(A)\ar[d]^{c_A}\\
P(A)\ar[r]^{\iota_{P(A)}}&P^2(A)&,
}
\xymatrix{P(A)&P^2(A)\ar[l]_{\delta^0_{P(A)}}\ar[r]^{P(\delta^0_{A})}&P(A)\\
&P(A)\ar[ul]^{\iota_A\delta^0_A}\ar[ur]_{\iota_A\delta^0_A}\ar[u]^{c_A}&.
}
$$
\end{enumerate} 
\end{defi}

\begin{lem}\label{coprhomo}
For all $A\in \Cc$ the map $c_A$ is a homotopy from $\iota_A\delta^0_A$ to $1_{P(A)}$.
\end{lem}
\begin{proof}
It follows from the definition.
\end{proof}

\begin{defi}[\cite{KP}, Def. I.4.7]\label{intercanvi}
An \textit{interchange} of $P$ is a natural automorphism $\mu:P^2\to P^2$ 
such that $\delta^k_{P(A)}\mu_A=P(\delta^k_A)$ and $P(\delta^k_A)\mu_A=\delta^k_{P(A)}$
for $k=0,1$.
\end{defi}

\begin{defi}\label{foldingmap}
A \textit{folding map} of $P$ is a natural transformation $\nabla:P^2\to P$
such that 
$\delta^k_A\nabla_A=\delta^k_A\delta^k_{P(A)}$, for $k=0,1$, and $\nabla_A\iota_{P(A)}=1_{P(A)}$.
\end{defi}

The transformations defined so far give rise to other useful transformations.
In particular we shall use the dual abstract version of the map $I^3\to I$ given by $(t,s,l)\mapsto t(s+l-sl)$.
\begin{lem}\label{coproduct2}Let $P$ be a functorial path with a symmetry $\tau$ and a coproduct $c$.
There is a natural transformation $\hat{c}:P\to P^3$ satisfying
$$
\left\{\begin{array}{ll}
(i)&\delta^0_{P^2(A)}\hat{c}_A=P(\delta^0_{P(A)})\hat{c}_A=c_A\\
(ii)&\delta^1_{P^2(A)}\hat{c}_A=P(\delta^1_{P(A)})\hat{c}_A=\iota_{P(A)}\\
       \end{array}\right.
\,
\left\{\begin{array}{ll}
(iii)&P^2(\delta^0_{A})\hat{c}_A=\iota_{P(A)}\iota_A\delta^0_A\\
(iv)&P^2(\delta^1_{A})\hat{c}_A=\tau_{P(A)}P(\tau_A)c_A\tau_A\\
       \end{array}\right..
$$
\end{lem}
\begin{proof}
Let $c':P\to P^2$ be the natural transformation defined by $c'_A:=P(\tau_A)c_A\tau_A$.
Then $c'$ satisfies the same properties of the coproduct $c$, with the maps $\delta^0$ and $\delta^k$ interchanged.
Define $\hat{c}_A:=c'_{P(A)}\circ c_A$. The above identities follow from the naturality of $\tau$ and $c$.
\end{proof}

Denote by $\sim$ the congruence of $\Cc$ transitively generated by the homotopy relation:
$f\sim g$ if there is a chain of homotopies 
$f\simeq \cdots\simeq g.$

\begin{defi}
A morphism $f:A\to B$ of $\Cc$ is a \textit{homotopy equivalence} if 
there exists a morphism $g:B\to A$ satisfying $fg\sim 1_B$ and $gf\sim 1_A$.
\end{defi}
Denote by $\Ss$ the class of homotopy equivalences of $\Cc$. 
This class is closed by composition and contains all isomorphisms.

\begin{prop}\label{loccongruencia}
Let $\Cc$ be a category with a functorial path, together with a symmetry and a coproduct. 
Then the categories $\pi\Cc:=\Cc/\sim$ and $\Cc[\Ss^{-1}]$
are canonically isomorphic.
\end{prop}
\begin{proof}
By Lemma $\ref{coprhomo}$ the map $\iota_A:A\to P(A)$ is a homotopy equivalence for all $A\in\Cc$. The result follows from Proposition 1.3.3 of \cite{GNPR}.
\end{proof}

We will use the following constructions.
\begin{defi}\label{def_map_path}
\begin{enumerate}[(1)]
 \item  Assume that the pull-back diagram
$$
\xymatrix{\pb\ar[d]_{\pi_1}\Pp(f)\ar[r]^{\pi_2}&P(B)\ar[d]^{\delta^0_B}\\
A\ar[r]^f&B&.
}
$$
exists. Then $\Pp(f)$ is called the \textit{mapping path of $f$}.
\end{enumerate}
\begin{enumerate}[(2)]
 \item Assume that the pull-back diagram
$$
\xymatrix{\pb\ar[d]_{\pi_1}\Pp(f,f')\ar[r]^{\pi_2}&P(B)\ar[d]^{(\delta^0_B,\delta^1_B)}\\
A\times A'\ar[r]^{f\times f'}&B\times B&.
}
$$
exists. Then $\Pp(f,f')$ is called the \textit{double mapping path of $f$ and $f'$}.
\end{enumerate}
\end{defi}
With these notations we have $\Pp(f,1_B)=\Pp(f)$.

\subsection{Axioms for a P-category}
Let $\Cc$ be a category with finite products and a final object $e$. Assume that $\Cc$ has
a functorial path $P$, together with a symmetry $\tau$, 
an interchange $\mu$, a coproduct $c$ and a folding map $\nabla$.
Assume as well that $\Cc$ has
two distinguished 
classes of maps $\Ff$ and $\Ww$ called \textit{fibrations} and \textit{weak equivalences} respectively.
A map will be called a \textit{trivial fibration} if it is both a fibration and a weak equivalence.
As is customary, the symbol $\stackrel{\sim}{\lra}$ will be used for weak equivalences, while
$\twoheadrightarrow$ will denote a fibration.

\begin{defi}\label{BCE_cat}
The quadruple $(\Cc,P,\Ff,\Ww)$ is called a \textit{P-category} if the following axioms are satisfied:
\begin{enumerate}
 \item [(P$_1$)]  The classes $\Ff$ and $\Ww$ contain all isomorphisms and are closed by composition. The class $\Ww$ satisfies
the two out of three property.
The map $A\to e$ is a fibration for every $A\in\Cc$.
\item [(P$_2$)] The map $\iota_A:A\to P(A)$ is a weak equivalence and $(\delta^0_A,\delta^1_A):P(A)\to A\times A$ is a fibration.
The maps $\delta^0_A$ and $\delta^1_A$ are trivial fibrations, for every $A\in\Cc$.
 \item [(P$_3$)] 
Given a diagram $A\stackrel{u}{\to}C\stackrel{v}{\twoheadleftarrow}B$,
where $v$ is a fibration, the fibre product $A\times_CB$ exists, and the projection $\pi_1:A\times_CB\twoheadrightarrow A$ is a fibration.
In addition, if $v$ is a trivial fibration, then $\pi_1$ is so, and if $u$ is a weak equivalence,
then $\pi_2:A\times_CB\to B$ is also a weak equivalence.
\item [(P$_4$)]  The path preserves fibrations and weak equivalences and is compatible with the fibre product: $P(\Ff)\subset\Ff$, $P(\Ww)\subset\Ww$ and 
$P(A\times_CB)= PA\times_{PC}PB$.
 \item [(P$_5$)] For every fibration $v:A\twoheadrightarrow B$, the map $\ov{v}$ defined by the following diagram
is a fibration.
$$
\xymatrix{
P(A) \ar@/_/[ddr]_{(\delta^0_A,\delta^1_A)} \ar@/^/[drr]^{P(v)} \ar@{.>}[dr]|-{\ov{v}}\\
& \pb\Pp(v,v)\ar[d]\ar[r] & P(B) \ar[d]^{(\delta^0_{B},\delta^1_B)}\\
& A\times A \ar[r]_{v\times v} & B\times B
}
$$
\end{enumerate}
\end{defi}

\begin{rmk}
A category satisfying axioms (P$_1$) to (P$_3$) is a Brown category of fibrant objects
with a functorial path \cite{Br}.
Axiom (P$_5$) is dual to the \textit{relative cylinder axiom} of Baues \cite{Ba},
and can be described as a cubical 
homotopy lifting property in dimension 2 (see \cite{KP}, pag. 86).
Baues introduced P-categories as an abstract example of a fibration category. 
Although our notion of a P-category differs substantially from the notion introduced by Baues,
we borrow the same name. 
The axioms for Baues fibration categories are very similar to those of Anderson-Brown-Cisinski fibration categories,
the latter including conditions relative to limits, such as closure of fibrations under transfinite compositions. 
The motivation behind these additional axioms lies in the construction of homotopy colimits indexed by small diagrams.
We refer to \cite{Ba}, \cite{Ci} and \cite{RB} for details.
\end{rmk}
Our objective in this section is to study the homotopy category
$\Ho(\Cc):=\Cc[\Ww^{-1}].$
We first prove some useful results
that are a consequence of the axioms.

\begin{lem}[cf. \cite{Br}, Factorization Lemma]\label{factoritza_brown}
Let $f:A\to B$ be a morphism in a P-category category $\Cc$.
Define maps
$p:=\pi_1$, $q:=\delta^1_B\pi_2$, and $\iota:=(1_A,\iota_Bf)$.
Then the diagram
$$\xymatrix{
A&\Pp(f)\ar[l]_{p}\ar[r]^{q}&B\\
&A\ar@{=}[lu]\ar[ur]_f\ar[u]^{\iota}&
}$$
commutes. In addition:
\begin{enumerate}[(1)] 
 \item The map $p$ is a trivial fibration, and $q$ is a fibration.
 \item The map $\iota$ is a homotopy equivalence with homotopy inverse $p$.
 \item If $f$ is a weak equivalence, then $q$ is a trivial fibration.
\end{enumerate}
\end{lem}
\begin{proof}
Since $\delta^0_B$ is a fibration, the mapping path $\Pp(f)$ exists by (P$_3$). From the definitions it 
is immediate that the above diagram commutes.

Assertion (1) follows analogously to the proof of Brown's Factorization Lemma \cite{Br} in its dual version.
Assertion (3) follows from (1) and the two out of three property of $\Ww$.
We prove (2). Since $p\iota=1$, it suffices to define a homotopy from $\iota p$ to
the identity.
Let $h$ be the morphism defined by the following pull-back diagram:
$$
\xymatrix{
\Pp(f) \ar@/_/[ddr]_{\iota_A\pi_1} \ar@/^/[drr]^{c_B\pi_2} \ar@{.>}[dr]|-{h}\\
& \pb P(\Pp(f)) \ar[d] \ar[r] & P^2(B) \ar[d]^{P(\delta^0_{B})}\\
& P(A) \ar[r]_{P(f)} & P(B)&,
}
$$
where $c$ is the coproduct of the path (see Definition $\ref{coproducte}$) and satisfies $P(\delta^0_{B})c_B=\iota_B\delta^0_B$.
Using the naturality of $\iota$ one sees that the above solid diagram commutes. Hence $h$ is well defined.
By (P$_4$), the pull-back $P(\Pp(f))$ is a path object of $\Pp(f)$ and for $k=0,1$ we have
$$\delta^k_{\Pp(f)}h=(\delta^k_AP(\pi_1),\delta^k_{P(B)}P(\pi_2))(\iota_A\pi_1,c_B\pi_2)
=(\pi_1,\delta^k_{P(B)}c_B\pi_2).$$
Since $\delta^0_{P(B)}c_B=\iota_B\delta^0_B$ and $\delta^1_{P(B)}c_B=1_B$, it follows that
$
\delta^0_{\Pp(f)}h=(\pi_1,\iota_B\delta^0_B\pi_2)=(1_A,\iota_Bf)\pi_1=\iota p$ and 
$\delta^1_{\Pp(f)}h=(\pi_1,\pi_2)=1.
$
Hence $h$ is a homotopy from $\iota p$ to the identity.
\end{proof}

Axiom (P$_5$) states that for a fibration $v:A\twoheadrightarrow B$, the induced morphism
$\ov v:P(A)\to \Pp(v,v)$ is a fibration. We prove an analogous statement for weak equivalences.
\begin{lem}\label{induit_we}
Let $w:A\stackrel{\sim}{\to}B$ be weak equivalence in a P-category $\Cc$. Then the induced map
$$\ov{w}:=((\delta^0_A,\delta^1_A),P(w)):P(A)\to \Pp(w,w)$$
is a weak equivalence.
\end{lem}
\begin{proof}
By (P$_3$) the maps $1_A\times w$ and $w\times 1_B$ are weak equivalences. Therefore the composition
$w\times w$ is also a weak equivalence.
Consider the commutative diagram
$$
\xymatrix{
&\ar[dl]_{\ov{w}}P(A)\ar[d]^{P(w)}\\
\pb \Pp(w,w)\ar[r]^{\pi_2}\ar[d]_{\pi_1}&P(B)\ar[d]^{(\delta^0_B,\delta^1_B)}\\
A\times A\ar[r]^{w\times w}&B\times B&.
}
$$
By (P$_3$), the projection $\pi_2$
is a weak equivalence. By (P$_4$), the map $P(w)$ is a weak equivalence.
Hence $\ov{w}$ is a weak equivalence by the two out of three property of $\Ww$.
\end{proof}

\begin{lem}\label{profib}
The map $\pi_A$ defined by the following pull-back diagram
is a trivial fibration.
$$
\xymatrix{
P^2(A) \ar@/_/[ddr]_{\delta^0_{P(A)}} \ar@/^/[drr]^{P(\delta^1_A)} \ar@{.>}[dr]|-{\pi_A}\\
& \pb \Pp(\delta^1_A)\ar[d]\ar[r] & P(A) \ar[d]^{\delta^0_{A}}\\
& P(A) \ar[r]_{\delta^1_A} & A
}
$$
\end{lem}
\begin{proof}
Define a map $p_A$ via the commutative diagram:
$$
\xymatrix{
\Pp(\delta^1_A,\delta^1_A)\ar[dd]^{\pi_1} \ar@/^/[drr]^{\pi_2} \ar@{.>}[dr]|-{p_A}\\
& \pb \Pp(\delta^1_A)\ar[d]\ar[r] & P(A) \ar[d]^{\delta^0_{A}}\\
P(A)\times P(A)\ar[r]^-{\pi_1}& P(A) \ar[r]_{\delta^1_A} & A
}.
$$
This map is a base extension of
the trivial fibration $\delta^1_A:P(A)\to A$, by $\delta^1_A\pi_2:\Pp(\delta^1_A)\to A$.
Therefore by (P$_3$) it is a trivial fibration.
The map $\pi_A:P^2(A)\to \Pp(\delta^1_A)$ is the composition
$$P^2(A)\xra{\ov{\delta^1_A}}\Pp(\delta^1_A,\delta^1_A)\stackrel{p_A}{\lra}\Pp(\delta^1_A),$$
where $\ov{\delta^1_A}$ is a trivial fibration by (P$_5$) and Lemma $\ref{induit_we}$.
Therefore $\pi_A$ is a trivial fibration.
\end{proof}

\subsection{Cofibrant models}

\begin{defi}\label{defcof}
An object $C$ of a P-category $\Cc$ is called \textit{cofibrant} if for any solid diagram
$$\xymatrix{
&A\ar@{->>}[d]^{w}_{\wr}\\
C\ar@{.>}[ur]^g\ar[r]^f&B
}$$
in which $w$ is a trivial fibration, there exists a dotted arrow $g$ making the diagram commute.
\end{defi}

The following is a homotopy lifting property for cofibrant objects with respect to trivial fibrations.
\begin{lem}\label{holift}
Let $C$ be a cofibrant object of a P-category $\Cc$, and let $v:A\stackrel{\sim}{\twoheadrightarrow} B$ be a trivial fibration.
For every commutative solid diagram of $\Cc$
$$
\xymatrix{
C \ar@/_/[ddr]_{h} \ar@/^/[drrr]^{(f_0,f_1)} \ar@{.>}[dr]|-{\wt h}\\
& P(A) \ar[d]^{P(v)} \ar[rr]_{(\delta^0_A,\delta^1_A)} && A\times A \ar[d]^{v\times v}\\
& P(B) \ar[rr]_{(\delta^0_B,\delta^1_B)} && B\times B&,
}
$$
there exists a dotted arrow $\wt h$,
making the diagram commute.
In other words: every homotopy $h:vf_0\simeq vf_1$ lifts to a homotopy $\wt h:f_0\simeq f_1$
such that $P(v)\wt h=h$.

\end{lem}
\begin{proof}
Let $H=(f_0,f_1,h)$ and consider the solid diagram of $\Cc$
$$\xymatrix{
&&P(A)\ar@{->>}[d]^{\ov{v}}_{\wr}\\
C\ar@{.>}[urr]^{\wt h}\ar[rr]^{H}&&\Pp(v,v)&.
}$$

By (P$_5$) and 
Lemma $\ref{induit_we}$ the map
$\ov{v}=((\delta^0_A,\delta^1_A),P(v))$ is a trivial fibration.
Since $C$ is cofibrant, there exists $\wt h$ such that $\ov{v}\wt h=H$.
Hence $(\delta^0_A,\delta^1_A)\wt h=(f_0,f_1)$ and $P(v)\wt h=h$.
\end{proof}

\begin{prop}[cf. \cite{GM}, Corollary 10.7]\label{relequiv}
The homotopy relation in a P-category is an equivalence relation for morphisms whose source is cofibrant. 
\end{prop}
\begin{proof}
By Lemma $\ref{reflexiva}$ the homotopy relation is reflexive. By Lemma $\ref{simetrica}$ it is symmetric.
We prove transitivity. Let $C$ be a cofibrant object and let $f,f',f'':C\to A$ be morphisms
 together with homotopies $h:f\simeq f'$ and $h':f'\simeq f''$.
Consider the solid diagram
$$
\xymatrix{
&&P^2(A)\ar@{->>}[d]^{\pi_A}_{\wr}\\
C\ar@{.>}[rru]^\Ll\ar[rr]_{(h,h')}&&\Pp(\delta^1_A)&.
}
$$
By Lemma $\ref{profib}$ the map $\pi_A=(\delta^0_{P(A)},P(\delta^1_A))$ is a trivial fibration.  Since $C$ is cofibrant,
there exists a dotted arrow $\Ll$
such that $\pi_A\Ll=(h,h')$. Therefore $\delta^0_{P(A)}\Ll=h$ and $P(\delta^1_A)\Ll=h'$.\\
Let $h\wt{+}h':=\nabla_A\Ll$,
where $\nabla$ is the folding map (see Definition $\ref{foldingmap}$) and satisfies
$\delta^k\nabla=\delta^k\delta^k=\delta^kP(\delta^k)$.
Then $\delta^0_A(h\wt{+}h')=f$ and $\delta^1_A(h\wt{+}h')=f''$. Hence
$h\wt{+}h':f\simeq f''$.
\end{proof}

Given a cofibrant object $C$ of a P-category $\Cc$, denote by
$[C,A]:=\Cc(C,A)/\simeq $
the class of maps from $C$ to $A$ modulo homotopy.
Denote by $\Cc_{cof}$ the full subcategory of cofibrant objects of $\Cc$ and by
$\pi\Cc_{cof}$ the quotient category $\pi\Cc_{cof}(A,B)=[A,B]$.

\begin{prop}[cf. \cite{GM}, Theorem 10.8]\label{bijecciocofibrants}
Let $C$ be a cofibrant object of a P-category $\Cc$. Every weak equivalence $w:A\stackrel{\sim}{\lra} B$ of $\Cc$
induces a bijection
$w_*:[C,A]\lra [C,B]$.
\end{prop}
\begin{proof}
We first prove surjectivity. Let $w:A\to B$ be a weak equivalence.
By Lemma $\ref{factoritza_brown}$, for every morphism $f:C\to B$ we have a solid diagram
$$
\xymatrix{
&A\ar@<.6ex>[d]^{\iota}\ar@/^2pc/[dd]^{w}\\
&\Pp(w)\ar@<.6ex>[u]^{p}\ar@{->>}[d]^{q}_{\wr}\\
C\ar@{.>}[ur]^{g'}\ar[r]^{f}&B&,
}
$$
where $q$ is a trivial fibration, $q \iota=w$ and $\iota p\simeq 1$. Since $C$ is cofibrant,
there exists $g'$ such that $q g'=f$. Let $g:=p g'$.
Then
$wg=q\iota p g'\simeq q g'=f$. Therefore $[wg]=[f]$, and $w_*$ is surjective.

To prove injectivity, let $f_0,f_1:C\to B$ be two morphisms of $\Cc$ such that $h:wf_0\simeq wf_1$.
Let $H=(f_0,f_1,h)$ and consider the solid diagram
$$\xymatrix{
&&P(A)\ar[d]^{\ov{w}}_{\wr}\\
C\ar@{.>}[urr]^{G}\ar[rr]^-{H}&&\Pp(w,w)&.
}$$
By Lemma $\ref{induit_we}$ the map $\ov{w}=((\delta^0_A,\delta^1_A),P(w))$ is a weak equivalence.
Since $\ov{w}_*$ is surjective,
there exists a dotted arrow $G$ such that $\ov{w}G\simeq H$.
It follows that
$f_0\simeq \delta^0_AG\simeq \delta^1_A G\simeq f_1$. Then $f_0\simeq f_1$ by Lemma $\ref{relequiv}$.
\end{proof}
It follows from Proposition $\ref{bijecciocofibrants}$, that if a P-category $\Cc$ has enough cofibrant models,
then the triple $(\Cc,\Ss,\Ww)$ is a left Cartan-Eilenberg category. In particular the inclusion
induces an equivalence of categories $\pi\Cc_{cof}\stackrel{\sim}{\lra}\Ho(\Cc)$.
We prove this statement in a more general situation, that of a subcategory of a P-category having enough cofibrant models.

\begin{teo}\label{Pcat_es_CE}
Let $(\Cc,P,\Ff,\Ww)$ be a P-category and
let $\Dd$ be a full subcategory of $\Cc$ such that:
\begin{enumerate}[(i)]
\item The mapping path $\Pp(f)$ of a morphism $f:A\to B$ between objects of $\Dd$ is an object of $\Dd$.
\item There is a full subcategory $\Mm\subset\Dd\cap \Cc_{cof}$ such that
for every object $D$ of $\Dd$ there exists an object $M\in\Mm$ together with a
weak equivalence $M\stackrel{\sim}{\lra} D$.
\end{enumerate}
Then the triple $(\Dd,\Ss,\Ww)$ is a left Cartan-Eilenberg category and $\Mm$ is a full subcategory of cofibrant models. The inclusion induces an
equivalence of categories
$\pi \Mm\stackrel{\sim}{\lra}\Ho(\Dd).$
\end{teo}
\begin{proof}
By (i), the functorial path of $\Cc$ restricts to a functorial path in $\Dd$.
By (P$_2$) the class $\Ss$ of strong equivalences is contained in the saturation of $\Ww$.
Hence the triple $(\Dd,\Ss,\Ww)$ is a category with strong and weak equivalences.
By Proposition $\ref{bijecciocofibrants}$ every object in $\Mm$ is Cartan-Eilenberg cofibrant in $(\Dd,\Ss,\Ww)$.
Hence by (ii), the triple $(\Dd,\Ss,\Ww)$ is a Cartan-Eilenberg category with left models in $\Mm$.
By Theorem 2.3.4 of \cite{GNPR}, the inclusion induces an equivalence of categories
$\Mm[\Ss^{-1},\Dd]\stackrel{\sim}{\lra} \Ho(\Dd)$, where $\Mm[\Ss^{-1},\Dd]$ denotes the full subcategory of $\Dd[\Ss^{-1}]$ whose objects are in $\Mm$.
The equivalence $\pi\Mm\stackrel{\sim}{\lra}\Mm[\Ss^{-1},\Dd]$ follows from
Proposition $\ref{loccongruencia}$.
\end{proof}

\section{Examples of P-categories}

The archetypal example of a P-category is given by the category of topological spaces, with 
the weak equivalences being
continuous maps that induce isomorphisms on all homotopy groups.
From the algebraic side, the basic example is given by the category of differential graded algebras over a field of characteristic zero, 
with the class of weak equivalences defined by those morphisms inducing isomorphisms on the cohomology groups.
In this section we present both examples in detail. We also provide a criterion of structure transfer
and apply it to two fundamental examples appearing in mixed Hodge theory:
 the categories of filtered and bifiltered differential graded algebras.

\subsection{Topological spaces}\label{topologics}
Let $I=[0,1]\subset \RR$ denote the unit interval.
Given a topological space $X$, let $P(X):=X^I$ be the set of all maps
$\sigma:I\to X$ with the compact open topology.
Define structural maps $\iota_X:X\to P(X)$ and $\delta^k:P(X)\to X$
by $\iota_X(x)(t)=x$, and $\delta^k_X(\sigma)=\sigma(k)$, for $k=0,1$.
The structure for the functorial path $P$ (symmetry, coproduct, interchange and folding map) is obtained
from the maps of $\ref{cilindre_top}$, through the bijection
$\mathsf{Top}(X,P(Y)){\rightleftarrows}\mathsf{Top}(X\times I,Y)$.

\begin{defi}\label{serrefib}
A morphism $v:X\to Y$ of topological spaces is called \textit{Serre fibration} if
for every $n\geq 0$, and
any commutative diagram
$$
\xymatrix{
U\ar[d]_{i_0}\ar[r]^f&X\ar[d]^v\\
U\times I\ar@{.>}[ur]^{H}\ar[r]_-G&B&,
}
$$
where $U$ is the unit disk of $\RR^n$,
a dotted arrow $H$ exists, making the diagram commute.
\end{defi}

\begin{defi}
A map $w:X\to Y$ of topological spaces is called \textit{weak homotopy equivalence} if
the induced map $w_*:\pi_0(X)\to \pi_0(Y)$ is a bijection and $w_*:\pi_n(X,x)\to \pi_n(Y,w(x))$ is an isomorphism
for every $x\in X$ and every $n\geq 1$.
\end{defi}

With the above definitions it follows that every CW-complex is cofibrant in the sense of Definition $\ref{defcof}$. We have:
\begin{prop}
The category $\mathsf{Top}$ of topological spaces with the classes
$\Ff=\{\text{Serre fibrations}\}$ and $\Ww=\{\text{weak homotopy equivalences}\}$
and the functorial path $P(X)=X^I$ is a P-category.
The inclusion induces an equivalence of categories
$\pi\mathsf{CW}\stackrel{\sim}{\lra}\Ho(\mathsf{Top}).$
\end{prop}
\begin{proof}
Axioms (P$_1$) to (P$_4$) are standard. A proof of (P$_5$) can be found in \cite{Ba2}, pag. 133.
Every CW-complex is cofibrant,
and every space is weakly equivalent to a CW-complex (see for example \cite{Q1}, \cite{DS}, \cite{Hov}).
The equivalence follows from Theorem $\ref{Pcat_es_CE}$.
\end{proof}

\subsection{Differential graded algebras}\label{secciodgas}
Consider the category $\dga{}{\kk}$ of dga's over a field $\kk$ of characteristic 0.
The field $\kk$ is the initial object, and 0 is the final object.
The functorial path is defined by
$$P(A)=A[t,dt]=A\otimes(t,dt)\text{ and }P(f)=f\otimes 1,$$
together with structural maps
$\iota_A=1_A\otimes 1$, and $\delta^k_A(a(t))=a(k)$, for $k=0,1$.
For $n\geq 1$,
$$P^n(A)=A[t_1,dt_1,\cdots,t_ndt_n]=A\otimes(t_1,dt_1)\otimes\cdots\otimes(t_n,dt_n).$$
The following maps are defined dually to the maps of $\ref{cilindre_top}$.
\begin{enumerate}[1.]
 \item  Symmetry. Define $\tau_A:A[t,dt]\to A[t,dt]$  by $t\mapsto 1-t$.
 \item   Interchange. Define $\mu_A:A[t,dt,s,ds]\to A[t,dt,s,ds]$ by $t\mapsto s$ and $s\mapsto t$.
 \item   Coproduct. Define  $c_A:A[t,dt]\to A[t,dt,s,ds]$  by $t\mapsto ts$,
 \item  Folding map. Define  $\nabla_A:A[t,dt,s,ds]\to A[t,dt]$ by $t\mapsto t$ and $s\mapsto t$.
\end{enumerate}

Denote by $\dga{0}{\kk}$ the category of \textit{0-connected dga}'s:
these are the dga's $(A,d)$ such that the unit $\eta:\kk\to A$ 
induces an isomorphism $\kk\cong H^0(A)$. By Sullivan's theory of minimal models
 (\cite{Su}, see also \cite{BG}, \cite{GM} or \cite{FHT}), every 0-connected dga has a Sullivan minimal model, and 
Sullivan minimal dga's are cofibrant.
Denote by $\Sdga{}{\kk}$ the full subcategory of Sullivan minimal dga's. We have:

\begin{prop}\label{dga_pcat}
The category $\dga{}{\kk}$ with
$\Ff=\{\text{surjections}\}$ and $\Ww=\{\text{quasi-isomorphisms}\}$,
and the functorial path $P(A)=A[t,dt]$, is a P-category.
The inclusion induces an equivalence of categories
$\pi\Sdga{}{\kk}\stackrel{\sim}{\lra}\Ho\left(\dga{0}{\kk}\right).$
\end{prop}
\begin{proof}
The only non-trivial axiom is (P$_5$). The double mapping path of $v:A\twoheadrightarrow B$ is
$$\Pp(v,v)=\left\{(a_0,a_1,b(t))\in A\times A\times B[t,dt]; b(i)=v(a_i)\right\},$$
and the map $\ov v:A[t,dt]\lra\Pp(v,v)$ is given by
$\ov v(a(t))=(a(0),a(1),(v\otimes 1)(a(t)).$

Let $(a_0,a_1,b(t))\in \Pp(v,v)$. Since $v\otimes 1$ is surjective, there exists an element $\wt b(t)\in A[t,dt]$ such that
$(v\otimes 1)\wt b(t)=b(t)$. Let
$a(t):=(a_0-\wt b(0))(1-t)+(a_1-\wt b(1))t+\wt b(t).$
Then $\ov v(a(t))=(a_0,a_1,b(t))$. Therefore $\ov v$ is surjective, and (P$_5$) is satisfied.

By Proposition 7.7 of \cite{BG}, every 0-connected dga has a Sullivan minimal model.
By Proposition 6.4 of loc. cit. every Sullivan minimal dga is cofibrant.
The equivalence of categories follows from Theorem $\ref{Pcat_es_CE}$
with $\Cc=\dga{}{\kk}$, $\Dd=\dga{0}{\kk}$ and $\Mm=\Sdga{}{\kk}$.
\end{proof}

\subsection{Transfer of structures}

Let $\Cc$ be a category with finite products and a final object. Assume that $\Cc$ has
a functorial path $P$, together with a symmetry, 
an interchange, a coproduct and a folding map.
\begin{lem}\label{transfer_pcat} 
Let $(\Dd,P,\Ff,\Ww)$ be a P-category and $T:\Cc\to \Dd$ a functor such that:
\begin{enumerate}[(i)]
\item  For every object $A\in\Cc$, $T(P(A))=P(T(A))$, $T(\iota_A)=\iota_{T(A)}$, and $T(\delta^k_A)=\delta^k_{T(A)}$.
\item  Given morphisms $A\stackrel{u}{\to}C\stackrel{v}{\twoheadleftarrow}B$ of $\Cc$,
where $T(v)$ is a fibration, the fibre product $A\times_C B$ exists, and satisfies
$P(A\times_CB)=PA\times_{P(C)}P(B)$ and $T(A\times_CB)=T(A)\times_{T(C)}T(B).$
\end{enumerate}
Then the tuple $(\Cc,P,T^{-1}(\Ff),T^{-1}(\Ww))$ is a P-category.
\end{lem}
\begin{proof}
Axiom (P$_1$) is trivial. Axiom (P$_2$) follows from (i). The rest follows from (ii).
\end{proof}

\subsection{Filtered differential graded algebras}\label{fdgas}
Denote by $\Fdga{}{\kk}$ the category of filtered dga's over $\kk$.
The base field $\mathbf k$ is considered as a filtered
dga with the trivial filtration and the unit map $\eta:\kk\to A$ is filtered.
We will restrict to filtered dga's $(A,W)$ whose filtration is regular and exhaustive:
for each $n\geq 0$ there exists $q\in\ZZ$ such that $W_qA^n=0$, and $A=\cup_pW_pA$.

The spectral sequence associated with a filtered dga $A$ is compatible with the multiplicative structure. 
Hence for all $r\geq 0$, the term $E_r^{*,*}(A)$
is a bigraded dga with differential $d_r$ of bidegree $(r,1-r)$.
For the rest of this section we fix an integer $r\geq 0$. We adopt the following definition of \cite{HT} 
(see also \cite{CG1}).

\begin{defi}
A morphism of filtered dga's $f:A\to B$ is called \textit{$E_r$-quasi-isomorphism} (resp. \textit{$E_r$-surjection}) if
 $E_r(f):E_r(A)\to E_r(B)$ is a quasi-isomorphism (resp. surjective).
\end{defi}

\begin{nada}
Let $\Lambda(t,dt)$ be the free dga with generators $t$ and $dt$ of degree $0$ and $1$ respectively.
For $r\geq 0$, define an increasing filtration $\sigma[r]$ on $\Lambda(t,dt)$ by letting
$t$ be of weight $0$ and $dt$ of weight $-r$ and 
extending multiplicatively. Note that $\sigma[0]$ is the trivial filtration, and $\sigma[1]$ is the b\^{e}te filtration.
\end{nada}
\begin{defi}\label{rpath}
The \textit{$r$-path} $P_r(A)$ of a filtered dga $A$ is the dga $A\otimes\Lambda(t,dt)$ with the filtration defined by the convolution of
$W$ and $\sigma[r]$. We have:
$$W_pP_r(A)=\sum_q W_{p-q}A\otimes \sigma[r]_q\Lambda(t,dt)=(W_{p}A\otimes\Lambda(t))\oplus (W_{p+r}A\otimes\Lambda(t)dt).$$
\end{defi}

\begin{prop}\label{rpstruc_fdga}
The category $\Fdga{}{\kk}$ together with the classes
$\Ff_r=\{\text{$E_r$-surjections}\}$ and $\Ee_r=\{\text{$E_r$-quasi-isomorphisms}\}$
and the path $P_r(A)$ is a P-category.
\end{prop}
\begin{proof}
We show that the functor
$E_r:\Fdga{}{\kk}\to \dga{}{\kk}$
satisfies the conditions of Lemma $\ref{transfer_pcat}$.
The isomorphism $E_r(P_r(A))\cong E_r(A)\otimes\Lambda(t,dt)$
gives the compatibility of $E_r$ with the $r$-path.
Consider filtered morphisms $(A,F)\stackrel{u}{\to}(C,F)\stackrel{v}{\twoheadleftarrow}(B,F)$,
where $v$ is an $E_r$-fibration. 
Then
$$A\times_CB=\Ker\left(A\times B\stackrel{u-v}{\lra}C\right).$$
Since $E_r(v)$ is surjective, by Proposition 1.1.11 of \cite{DeHII}, the map $v$ is strictly compatible with filtrations
and hence $u-v$ is so. Therefore
$E_r\Ker(u-v)=\Ker E_r(u-v)$
and $E_r$ is compatible with fibre products.
The result follows from Lemma $\ref{transfer_pcat}$ together with Proposition $\ref{dga_pcat}$.
\end{proof}

For the applications to mixed Hodge theory we shall be interested in the P-category structure of $\Fdga{}{\kk}$ associated with
 $E_1$-quasi-isomorphisms.
Likewise, in the category $\FFdga{}{\kk}$ of bifiltered dga's we shall consider the P-category structure associated with the class of
$E_{1,0}$-quasi-isomorphisms.

\begin{defi}
A morphism $f:A\to B$ of bifiltered dga's is called 
\textit{$E_{1,0}$-quasi-isomorphism} (resp. \textit{$E_{1,0}$-surjection}) if for all $p\in\ZZ$,
$E_1(Gr^p_Ff,W)$ is a quasi-isomorphism (resp. surjective).
\end{defi}

\begin{defi}\label{r0path}
The \textit{$(1,0)$-path object} of a bifiltered dga $A$ is 
the bifiltered dga 
defined by 
$$W_pF^qP_{1,0}(A)=(W_{p}F^q A\otimes\Lambda(t))\oplus (W_{p+1} F^qA\otimes\Lambda(t)dt).$$
\end{defi}

\begin{prop}\label{rpstruc_ffdga}
The category $\FFdga{}{\kk}$ together with the classes
$\Ff_{1,0}=\{\text{$E_{1,0}$-surjections}\}$ and $\Ee_{1,0}=\{\text{$E_{1,0}$-quasi-isomorphisms}\}$
and the path $P_{1,0}(A)$ is a P-category.
\end{prop}
\begin{proof}
 The proof is analogous to that of Proposition $\ref{rpstruc_fdga}$.
\end{proof}

\section{Diagram Categories}
We study the homotopy theory of diagram categories.
It is quite straightforward, that if the vertex categories are endowed with compatible P-category structures,
there is a P-category structure on the diagram category defined level-wise.
However, in general, level-wise cofibrant models do not give cofibrant models of the diagram category.
In the context of Cartan-Eilenberg categories,
we show that if we consider a weaker class of strong equivalences than the one associated with the functorial path,
one obtains a Cartan-Eilenberg structure on the diagram category with level-wise weak equivalences,
which is still useful to describe the morphisms 
in the homotopy category in terms of certain homotopy classes of morphisms up to homotopy between level-wise cofibrant objects.

\subsection{Level-wise P-category structure}
\begin{defi}\label{catdiagrames}
Let
$\Cc:I\to\mathsf{Cat}$ be a functor from a small category $I$, to the category of categories $\mathsf{Cat}$.
For all $i\in I$, denote $\Cc_i:=\Cc(i)\in \mathsf{Cat}$, and $u_*:=\Cc(u)\in \mathsf{Fun}(\Cc_i,\Cc_j)$,
for all $u:i\to j$.
The \textit{category $\GCc$ of diagrams associated with the functor $\Cc$} is defined as follows:
\begin{enumerate}[$\bullet$]
\item An object $A$ of $\GCc$ is given by a family of objects $\{A_i\in \Cc_i\}$, for all $i\in I$,
together with a family of morphisms $\varphi_u:u_*(A_i)\to A_j$, called \textit{comparison morphisms}, for every map $u:i\to j$.
Such an object is denoted as
$$A=\left(A_i\stackrel{\varphi_u}{\dashrightarrow}A_j\right).$$
\item A morphism $f:A\to B$ of $\GCc$ is a family of morphisms $\{f_i:A_i\to B_i\}$ of $\Cc_i$,
for all $i\in I$, such that for every map $u:i\to j$ of $I$, the diagram
$$\xymatrix{
u_*(A_i)\ar[d]_{u_*(f_i)}\ar[r]^{\varphi_u}&A_j\ar[d]^{f_j}\\
u_*(B_i)\ar[r]^{\varphi_u}&B_j&
}$$
commutes in $\Cc_j$. Denote $f=(f_i):A\to B$.
\end{enumerate}
By an abuse of notation, we will omit the notation of the functors $u_*$ and write $A_i$ for $u_*(A_i)$ and $f_i$ for $u_*(f_i)$,
whenever there is no danger of confusion.
\end{defi}

\begin{rmk}The category of diagrams $\GCc$ associated with $\Cc$ is the category of sections of the projection functor $\pi:\int_I\Cc\to I$,
where $\int_I\Cc$ is the Grothendieck construction of $\Cc$ (see \cite{Th}). 
If $\Cc:I\to\mathsf{Cat}$ is the constant functor $i\mapsto\Cc$ and
$\Cc(u)$ is the identity functor of a category $\Cc$, for all $u:i\to j$, then
$\GCc=\Cc^I$ is the diagram category of objects of $\Cc$ under $I$.
\end{rmk}

\begin{nada}\label{indexcat}
We will restrict our study of diagram categories satisfying the following axioms:
\begin{enumerate}
 \item [(I$_1$)] The index category $I$ is finite and has a 
\textit{degree function} $|\cdot|:\text{Ob}(I)\lra \{0,1\}$ such that
$|i|<|j|$ for every non-identity morphism
$u:i\to j$ of $I$.
 \item [(I$_2$)]For all $i\in I$, the category
$\Cc_i$ is equipped with a functorial path $P$, together with two classes $\Ff_i$ and $\Ww_i$
such that $(\Cc_i,P,\Ff_i,\Ww_i)$ is a P-category.
\item [(I$_3$)]For all $u:i\to j$ the functor $u_*:\Cc_i\to \Cc_j$ preserves path objects, fibrations, weak equivalences and fibre products.
\end{enumerate}
\end{nada}
A category $I$ satisfying (I$_1$) is a particular case of a Reedy category for which $I^+=I$.
The main examples of such categories are given by finite zig-zags
$$
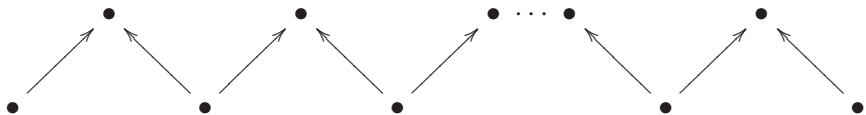
\xymatrix{
&\bullet&&\bullet&&\bullet\\
\bullet\ar[ur]&&\bullet\ar[ul]\ar[ur]&&\bullet\ar[ul]\ar[ur]
}\cdots\xymatrix{
\bullet&&\bullet\\
&\bullet\ar[ul]\ar[ur]&&\bullet\ar[ul]
}
$$
but other diagram shapes are admitted. For example:
$$
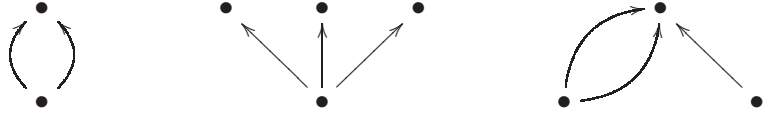
\xymatrix{
\bullet\\
\bullet\ar@/_1pc/[u]\ar@/^1pc/[u]
}\hspace{2cm}
\xymatrix{
\bullet&\bullet&\bullet\\
&\ar[ul]\ar[u]\ar[ur]\bullet
}\hspace{1.5cm}
\xymatrix{
&\bullet\\
\bullet\ar@/_1pc/[ur]\ar@/^1pc/[ur]& &\bullet\ar[ul]
}
$$
All objects at the bottom of the diagrams have degree 0, and the ones at the top have degree 1.

\begin{defi}
A morphism $f:A\to B$ in $\GCc$ is called \textit{weak equivalence} (resp. \textit{fibration}) if for all $i\in I$, the maps $f_i$ are weak equivalences
 (resp. fibrations) of $\Cc_i$. Denote by $\Ww$ (resp. $\Ff$) the class of weak equivalences (resp. fibrations) of the diagram category $\GCc$.
\end{defi}
\begin{defi}
The \textit{path object} $P(A)$ of a diagram $A$ of $\GCc$ is the diagram defined by
$$P(A)=\left(P(A_i)\stackrel{P(\varphi_u)}{\dashrightarrow}P(Aj)\right).$$
\end{defi}
There are natural morphisms of diagrams
$$
\xymatrix{
A&P(A)\ar[l]_{\delta^0_A}\ar[r]^{\delta^1_A}&A\\
&A\ar[u]^{\iota_A}\ar@{=}[ul]_=\ar@{=}[ur]^=&,
}
$$
defined level-wise.
This defines a functorial path on $\GCc$.

Let $A\stackrel{u}{\to}C\stackrel{v}{\leftarrow}B$ be a diagram of $\GCc$,
and assume that for all $i\in I$, the fibre product $A_i\times_{C_i}B_i$ exists.
Then the fibre product $A\times_CB$ is given level-wise by
$(A\times_BC)_i=A_i\times_{B_i}C_i.$
The comparison morphism $\psi_u:(A\times_CB)_i\to (A\times_CB)_j$ is given by
$\psi_u=(\varphi_u^A\pi_1,\varphi_u^B\pi_2)$, where $\varphi_u^A$ and $\varphi_u^B$ denote the comparison morphisms of $A$ and $B$ respectively.
\begin{prop}\label{diags_BCE}Let $\GCc$ be a diagram category satisfying conditions
{\normalfont(I$_1$)-(I$_3$)} of $\ref{indexcat}$.
Then $\GCc$ is a P-category with path objects, fibrations, weak equivalences and fibre products defined level-wise.
\end{prop}
\begin{proof}The functor $\GCc\to\Pi_{i\in I} \Cc_i$ induced by the inclusion $I_{dis}\to I$
satisfies the conditions of Lemma $\ref{transfer_pcat}$.
\end{proof}

Let $\Ss$ denote the class of homotopy equivalences of $\GCc$.
If $f=(f_i)$ is in $\Ss$, then $f_i\in \Ss_i$. In particular, since $\Ss_i\subset \Ww_i$ we have $\Ss\subset\Ww$.
Hence the triple $(\GCc,\Ss,\Ww)$
is a category with strong and weak equivalences. Our objective is to study the homotopy category $\Ho(\GCc):=\GCc[\Ww^{-1}]$.
Note that in general, this differs from the category of diagrams $\Gamma\Ho(\Cc)$ associated with the level-wise localized categories.

\subsection{Morphisms up to homotopy}
The characterization and existence of cofibrant models of a diagram category $\GCc$
involves a rectification of homotopy commutative morphisms.
We solve this problem by studying the factorization
of morphisms commuting up to fixed homotopies into the composition
 of morphisms in a certain localized category  $\GCc[\Hh^{-1}]$, with $\Ss\subset\Hh\subset\Ww$.
The following is a simple example illustrating the procedure that we will conduct.
\begin{example}[Model of a morphism of dga's]\label{model_morfisme_dga}
Consider a diagram of morphisms of dga's
$$\xymatrix{\ar@{=>}[rd]^F
\ar[d]_{f_0}A_0\ar[r]^{\varphi}&A_1\ar[d]^{f_1}\\
B_0\ar[r]^{\varphi}&B_1\\
}$$
where $F:A_0\to B_1[t,dt]$ is a homotopy from $f_1\varphi$ to $\varphi f_0$.
Consider the mapping path
$$\Pp(f_i)=\{(a,b(t))\in A_i\times B_i[t,dt]; f_i(a)=b(0)\},\,i=0,1$$
and define morphisms $p_i:\Pp(f_i)\to A_i$ and $q_i:\Pp(f_i)\to B_i$ by letting 
$p_i(a,b(t))=a$, and
$q_i(a,b(t))=b(i)$. The maps $q_i$ and $p_i$ are quasi-isomorphisms of dga's, for $i=0,1$.
Define a morphism $\psi:\Pp(f_0)\to \Pp(f_1)$ by letting
$\psi(a,b(t))=(\varphi(a),F(a))$. Then the diagram
$$
\xymatrix{
A_0\ar[r]^{\varphi}&A_1\\
\ar[u]^{p_0}\ar[d]_{q_0}\Pp(f_0)\ar[r]^{\psi}&\Pp(f_1)\ar[u]_{p_1}\ar[d]^{q_1}\\
B_0\ar[r]^{\varphi}&B_1
}
$$
commutes. The key point of this construction
resides in the 
definition of the morphism $\psi$ (which depends on the homotopy $F$, and only on the first variable), 
and the morphisms $q_i$ (whose definition depends on whether the index $i$ is a source or a target in the index category).
Note also that the morphisms $p_i$ are homotopy equivalences.
\end{example}

\begin{defi}\label{homorfismes}
A \textit{ho-morphism} $f:A\hto B$ between two objects of $\GCc$ is pair of families $f=(f_i,F_u)$ indexed by $i\in I$ and $u\in I(i,j)$, where:
\begin{enumerate}[(i)]
 \item  $f_i:A_i\to B_i$ is a morphism in $\Cc_i$, and
\item  $F_u: A_i\to P(B_j)$ is a morphism in $\Cc_j$ satisfying
$\delta^0_{B_j}F_u=f_j\varphi_u$ and $\delta^1_{B_j}F_u=\varphi_uf_i$.
Hence $F_u$ is a homotopy of morphisms of $\Cc_j$ making the diagram 
$$
\xymatrix{
A_i\ar@{=>}[dr]^F\ar[d]_{f_i}\ar[r]^{\varphi_u}&A_j\ar[d]^{f_j}\\
B_i\ar[r]_{\varphi_u}&B_j
}
$$
commute up to homotopy.
\end{enumerate}
\end{defi}
Denote by $\GCch(A,B)$ the set of ho-morphisms from $A$ to $B$.
Every morphism $f=(f_i):A\to B$ can be made into a ho-morphism $f=(f_i,F_u):A\hto B$ by letting $F_u=\iota_{B_j} f_j\varphi_u=\iota_{B_j} \varphi_uf_i$.
This defines an inclusion of sets
$$\GCc(A,B)\subset \GCch(A,B).$$

The composition of ho-morphisms is not well defined. This is due to the fact that the homotopy 
relation between objects of $\Cc_i$ is not transitive in general.
However, we can compose ho-morphisms with morphisms. The following is straightforward.
\begin{lem}\label{compo_senzilles}
Let $f:A\hto B$ be a ho-morphism and let $g:A'\to A$ and $h:B\to B'$ be morphisms of $\GCc$.
There are ho-morphisms $fg:A'\hto B$ and $hf:A\hto B'$,
given by
$$fg=(f_ig_i,F_ug_i),\text{ and }hf=(h_if_i,P(h_j)F_u).$$
If $f$ is a morphism, then $fg$ and $hf$ coincide with the standard composition of morphisms of $\GCc$.
\end{lem}

\begin{defi}\label{hohomotopies}
Let $f,g:A\hto B$ be two ho-morphisms. A \textit{homotopy from $f$ to $g$} is a ho-morphism $h:A\hto P(B)$
such that $\delta^0_Bh=f$ and $\delta^1_Bh=g$.
We use the notation $h:f\simeq g.$ 
\end{defi}
Equivalently, such a homotopy is given by a family of pairs $h=(h_i,H_u)$ where:
\begin{enumerate}[(i)]
 \item  $h_i:A_i\to P(B_i)$ satisfies $\delta^0_{B_i}h_i=f_i$ and $\delta^1_{B_i}h_i=g_i$, i.e., $h_i$ is a homotopy from $f_i$ to $g_i$.
\item  $H_u:A_i\to P^2(B_j)$ is a morphism in $\Cc_j$ satisfying
$$
\left\{\begin{array}{l}
P(\delta^{0}_{B_j}) H_u=F_u,\vspace{.2cm}\\
P(\delta^{1}_{B_j})H_u=G_u,
\end{array}\right.\text{ and }
\left\{\begin{array}{l}
\delta^{0}_{P(B_j)}H_u=h_j\varphi_u,\vspace{.2cm}\\
\delta^{1}_{P(B_j)}H_u=\psi_uh_i
\end{array}\right..
$$
\end{enumerate}

\begin{defi}\label{hoequivalencies}
A morphism $f:A\to B$ of $\GCc$ is said to be a \textit{ho-equivalence} if there exists a ho-morphism $g:B\hto A$ together with chains of homotopies of
ho-morphisms $gf\simeq\cdots \simeq 1_A$ and $fg\simeq \cdots \simeq 1_B$.
\end{defi}
Denote by $\Hh$ the closure by composition of the class of ho-equivalences. 
\begin{lem}
We have $\Ss\subset\Hh\subset \overline\Ww$. In particular, the triple $(\GCc,\Hh,\Ww)$ is a category with strong and weak equivalences.
\end{lem}
\begin{proof}
If $f$ and $g$ are homotopic morphisms of $\GCc$, then they are also homotopic as ho-morphisms. Therefore $\Ss\subset \Hh$.
If $f$ is a ho-equivalence, then
$f_i$ is a morphism of $\Ss_i$, for all $i\in I$. Since $\Ss_i\subset \overline\Ww_i$, it follows that $\Hh\subset\overline\Ww$.
\end{proof}

\subsection{Factorization of ho-morphisms}
Our next objective is to prove a Brown-type Factorization Lemma for ho-morphisms.

\begin{defi}
Let $f:A\hto B$ be a ho-morphism. The \textit{mapping path of $f$}
is the diagram
$$\Pp^h(f)=\left(\Pp(f_i)\stackrel{\psi_u}{\dashrightarrow}\Pp(f_j)\right),$$
where $\Pp(f_i)$ is the mapping path of $f_i$  
and the comparison morphism $\psi_u:\Pp(f_i)\to \Pp(f_j)$ is defined by
$\psi_u=(\varphi_u,F_u)\pi_1$ via the pull-back diagram
$$
\xymatrix{
\Pp(f_i)\ar[r]^{\pi_1}&A_i \ar@/_/[ddr]_{\varphi_u} \ar@/^/[drr]^{F_u} \ar@{.>}[dr]\\
&& \pb\Pp(f_j) \ar[d]^{\pi_1} \ar[r]_{\pi_2} & P(B_j) \ar[d]^{\delta^0_{B_j}}\\
&& A_j \ar[r]_{f_j} & B_j&.
}
$$
\end{defi}
\begin{rmk}
Since $\GCc$ is a P-category, every morphism $f:A\to B$ of $\GCc$ has a mapping path $\Pp(f)$.
We can consider $f$ as a ho-morphism, by letting $F_u=\iota f_j\varphi_u$, so that it has an associated mapping path 
$\Pp^h(f)$. The comparison morphisms of $\Pp(f)$ and $\Pp^h(f)$ do not coincide, but are only homotopic.
\end{rmk}

\begin{nada}
For $i\in I$, consider the first projection maps $p_{i}:=\pi_1:\Pp(f_i)\to A_i$.
For $u:i\to j$ we have
$$p_{j}\psi_u=\pi_1(\varphi_u,F_u)\pi_1=\varphi_u\pi_1=\varphi_up_{i}.$$
Therefore the family $p=(p_{i}):\Pp^h(f)\to A$ is a morphism of $\GCc$.

For $i\in I$, let ${q}_i:=\delta^{|i|}_{B_i}\pi_2:\Pp(f_i)\to B_i$, 
where $|i|\in\{0,1\}$ is the degree of $i$ (see condition (I$_1$) of $\ref{indexcat}$).
For $u:i\to j$ we have
$$q_{j}\psi_u=\delta^1_{B_j}\pi_2(\varphi_u,F_u)\pi_1=\delta^1_{B_j}F_u\pi_1=\varphi_uf_i\pi_1=\varphi_u\delta^0_{B_j}\pi_2=\varphi_uq_{i}.$$
Therefore the family $q=({q}_i):\Pp^h(f)\to B$ is a morphism of $\GCc$.

Note that $q$ is not defined level-wise via the Factorization Lemma $\ref{factoritza_brown}$
in which $q=\delta^1_B\pi_2$, but instead, we alternate between $\delta^0_B\pi_2$ and $\delta^1_B\pi_2$, depending on the degree of the index.
This needs to be done in order to obtain a morphism. As a result, $q$
is not necessarily a level-wise fibration.

For $i\in I$, let $\iota_{i}=(1_{A_i},\iota_{B_i}f_i):A_i\to \Pp(f_i)$. Then
$\psi_u\iota_{i}=(\varphi_u,F_u)$ and $\iota_{j}\varphi_u=(\varphi_u,\iota_{A_j}f_j\varphi_u).$
We next define a homotopy
from $\psi_u\iota_{i}$ to $\iota_{j}\varphi_u$.
Let $J_{u}$ be the morphism defined by the following pull-back diagram:
$$
\xymatrix{
&A_i \ar@/_/[ddr]_{\iota_{A_j}\varphi_u} \ar@/^/[drr]^{c_{B_j}F_u} \ar@{.>}[dr]|-{J_{u}}\\
&& \pb P(\Pp(f_j)) \ar[d] \ar[r] & P^2(B_j) \ar[d]^{P(\delta^0_{B_j})}\\
&& P(A_j) \ar[r]_{P(f_j)} & P(B_j)&.
}
$$
The coproduct (see Definition $\ref{coproducte}$) satisfies
$P(\delta^0_{B_j})c_{B_j}=\iota_{B_j}\delta^0_{B_j}$. Hence
the above solid diagram commutes and the map $J_{u}$ is well defined.

By (P$_4$), the pull-back $P(\Pp(f_j))$ is a path object of $\Pp(f_j)$, with
$$\delta^k_{\Pp(f_j)}=(\delta^k_{A_j}P(\pi_1),\delta^k_{P(B_j)}P(\pi_2)),\text{ for }k=0,1.$$
Therefore we have
$$\delta^k_{\Pp(f_j)}J_{u}=(\delta^k_{A_j}\iota_{A_j}\varphi_u,\delta^k_{P(B_j)}c^0_{B_j}F_u)=
(\varphi_u,\delta^k_{P(B_j)}c^0_{B_j}F_u).$$
Since $\delta^0_{P(B_j)}c^0_{B_j}=\iota_{B_j}\delta^0_{B_j}$ and $\delta^1_{P(B_j)}c^0_{B_j}=1_{B_j}$, it follows that
$$
\delta^0_{\Pp(f_j)}J_{u}=(\varphi_u,\iota_{B_j}\delta^0_{B_j}F_u)=\iota_{j}\varphi_u\text{ and }
\delta^1_{\Pp(f_j)}J_{u}=(\varphi_u,F_u)=\psi_u\iota_{i}.
$$
Therefore the family $\iota=(\iota_{i},J_{u}):A\hto\Pp^h(f)$ is a ho-morphism.
\end{nada}

\begin{prop}\label{inversaho}
Let $f:A\hto B$ be a ho-morphism. The diagram
$$\xymatrix{
A&\Pp^h(f)\ar[l]_{p}\ar[r]^{q}&B\\
&A\ar@{=}[lu]_=\ar@{~>}[ur]_f^=\ar@{~>}[u]^{\iota}
}$$
commutes. In addition:
\begin{enumerate}[(1)]
\item The maps $p$ and $\iota$ are weak equivalences.
\item  There is a homotopy of ho-morphisms between $\iota p$ and the identity, 
making $p$ into a ho-equivalence.
\item  If $f$ is a weak equivalence, then $q$ is a weak equivalence.
\end{enumerate}
\end{prop}
\begin{proof}
From the definitions it is straightforward that the above diagram
commutes.

Let us prove (1). From axiom $(P_3)$, the map $p$ is a weak equivalence.
By the two out of three property, it follows that $\iota$ is also a weak equivalence.

To prove (2) we define a homotopy between
$\iota p=(\iota_{i}p_{i},J_{u}p_{i})$
and $1_{\Pp(f_i)}$ as follows.

For all $i\in I$, let $h_i:\Pp(f_i)\to P(\Pp(f_i))$ be the morphism of $\Cc_i$ defined by
$h_i=(\iota_{A_i}\pi_1,c^0_{B_i}\pi_2)$.
This is a homotopy from $\iota_ip_i$ to the identity morphism (see the proof of Lemma $\ref{factoritza_brown}$).

Let $\wt H_u$ be the morphism defined by the following pull-back diagram:
$$
\xymatrix{
A_i \ar@/_/[ddr]_{\iota_{P(A_j)}\iota_{A_j}\varphi_u} \ar@/^/[drr]^{\hat{c}_{B_j}F_u}\ar@{.>}[dr]|-{\wt H_u}\\
& \pb P^2(\Pp(f_j)) \ar[d]^{} \ar[r]_{} & P^3(B_j) \ar[d]^{P^2(\delta^0_{B_j})}\\
& P^2(A_j) \ar[r]_{P^2(f_j)} & P^2(B_j)&,
}
$$
where $\hat{c}$ is the transformation of Lemma $\ref{coproduct2}$ and satisfies
$P^2(\delta^0_{B_j})\hat{c}_{B_j}=\iota_{P(B_j)}\iota_{B_j}\delta^0_{B_j}$. The above solid diagram
commutes and the map $\wt H_u$ is well defined.

Let
$H_u:=\wt H_u\pi_1:\Pp(f_i)\to P^2(\Pp(f_j)).$
By (P$_4$), the fibre product $P^2(\Pp(f_j))$ is a double path object of $\Pp(f_j)$.
From the properties of $\hat{c}$ we have:
$$\left\{\begin{array}{lll}
\delta^0_{P(\Pp(f_j))}H_u=h_j\psi_u&P(\delta^0_{\Pp(f_j)})H_u=J_{F_u}p_{f_i}\\
\delta^1_{P(\Pp(f_j))}H_u=P(\psi_u) h_i&P(\delta^1_{\Pp(f_j)})H_u=\iota_{\Pp(f_j)}(\psi_u)
  \end{array}\right..
$$
Therefore the family $h=(h_i,H_u)$ is a homotopy from $\iota p$ to $1$.

Let us prove (3). Assume that $f$ is a weak equivalence. By $(i)$, the map $\iota$ is a weak equivalence.
By the two out of three property, $q$ is a weak equivalence.
\end{proof}

\begin{lem}\label{morf_de_homorfs}
Let $f:A\hto B$ and $g:A\hto C$ be two ho-morphisms and $\alpha:A\to C$ a morphism of $\GCc$
making the diagram
$$
\xymatrix{
&\ar@{~>}[dr]^gA\ar@{~>}[dl]_f&\\
B\ar[rr]^{\alpha}&&C
}
$$
commute. There is a morphism
$(\alpha)_*:\Pp^h(f)\lra\Pp^h(g)$
compatible with $p$, $q$, and $\iota$.
\end{lem}
\begin{proof}
Let $(\alpha_i)_*$ be the morphism defined by the pull-back diagram:
$$
\xymatrix{
&\Pp(f_i) \ar@/_/[ddr]_{\pi_1} \ar@/^/[drr]^{P(\alpha_i)\pi_2} \ar@{.>}[dr]|-{(\alpha_i)_*}\\
&& \pb\Pp(g_i) \ar[d]^{\pi_1} \ar[r]^{\pi_2} & P(C_i) \ar[d]^{\delta^0_{C_i}}\\
&& A_i \ar[r]^{g_i} & C_i&.
}
$$

Since $\varphi_u\alpha_i=\alpha_j\varphi_u$ and $P(\alpha_j)F_u=G_u$, 
the family $(\alpha)_*:=(\alpha_i)_*$ is a morphism of diagrams.
A simple verification shows that the following diagram commutes.
$$
\xymatrix{
\Pp^h(f)\ar[d]^{(\alpha)_*}&A\ar@{=}[d]\ar@{~>}[l]_-{\iota}&\Pp^h(f)\ar[d]^{(\alpha)_*}\ar[r]^{q}\ar[l]_{p}&B\ar[d]^\alpha\\
\Pp^h(g)&A\ar@{~>}[l]_-{\iota}&\Pp^h(g)\ar[r]^{q}\ar[l]_{p}&C
}
$$
\end{proof}

\begin{nada}[Factorization map]\label{def_phi}Given two objects $A,B$ of $\GCc$ define a map
$$\Phi_{A,B}:\GCch(A,B)\lra \GCc[\Hh^{-1}](A,B)$$
as follows.
Let $f:A\hto B$ be a ho-morphism. By
 Proposition $\ref{inversaho}$
we have $f=q_f\circ \iota_f$, where $q_f$ is a morphism of $\GCc$ and
$\iota_f$ is a homotopy inverse for a ho-equivalence $p_f$.
We then let
$$\Phi_{A,B}(f)=q_f\circ p_f^{-1}.$$
\end{nada}

We next prove some useful properties of the factorization map.

\begin{lem}\label{p1}
If $f:A\to B$ is a morphism of $\GCc$ then $\Phi_{A,B}(f)=f$ in $\GCc[\Hh^{-1}]$.
\end{lem}
\begin{proof}
Since $f$ is a morphism, the map $\iota_f:A\to \Pp^h(f)$ is also a morphism and the diagram
$$\xymatrix{A&\ar[l]_{p_f}\Pp^h(f)\ar[r]^{q_f}&B\\
&A\ar[u]^{\iota_f}\ar[ur]^{=}_f\ar@{=}[ul]_=&}$$
commutes. Hence $q_f\circ p_f^{-1}=f$ in $\GCc[\Hh^{-1}]$.
\end{proof}

\begin{lem}\label{p2}
Let $f:A\hto B$ be a ho-morphism, and let $g:B\to C$ be a morphism of $\GCc$. Then $\Phi_{B,C}(g)\circ \Phi_{A,B}(f)=\Phi_{A,C}(gf)$.
\end{lem}
\begin{proof}
By Lemma $\ref{morf_de_homorfs}$ we have a commutative diagram
$$
\xymatrix{
&\ar[dl]_{p_f}\Pp^h(f)\ar[r]^{q_f}\ar[d]^{(g)_*}&B\ar[d]^g\ar[rd]^g&\\
A&\ar[r]^{q_{gf}}\Pp^h(gf)\ar[l]_{p_{gf}}&C\ar@{=}[r]&C
}
$$
Hence $g\circ q_f\circ p_f^{-1}=q_{gf}\circ p_{gf}^{-1}$ in $\GCc[\Hh^{-1}]$. The result follows from Lemma $\ref{p1}$.
\end{proof}

\begin{lem}\label{p3}
Let $f,g:A\hto B$ be two ho-morphisms. If $f\simeq g$ then $\Phi_{A,B}(f)=\Phi_{A,B}(g)$.
\end{lem}
\begin{proof}
Let $h:A\hto P(B)$ be a homotopy from $f$ to $g$.
By Lemma $\ref{morf_de_homorfs}$ we have a diagram
$$
\xymatrix{
A\ar@{=}[d]&\Pp^h(f)\ar[l]_{p_f}\ar[r]^{q_f}&B\ar@{=}[d]\\
A\ar@{=}[d]&\ar[u]^{(\delta_B^0)_*}\ar[d]_{(\delta_B^1)_*}\Pp^h(h)\ar[l]_{p_h}\ar[r]^{\delta_B^0q_h}&B\ar@{=}[d]\\
A&\Pp^h(g)\ar[l]_{p_g}\ar[r]^{q_g}&B\\
}
$$
where every square commutes in $\GCc$, except for the lower-right one, which commutes in $\GCc[\Ss^{-1}]$.
Since $\Ss\subset \Hh$, the localization $\delta:\GCc\to \GCc[\Hh^{-1}]$
factors through $\gamma:\GCc\to \GCc[\Ss^{-1}]$.
Therefore $q_f\circ p_f^{-1}=q_g\circ p_g^{-1}$ in $\GCc[\Hh^{-1}]$.
\end{proof}

\subsection{Homotopy classes of ho-morphisms}
Denote by $\GCc_{cof}$ the full subcategory of $\GCc$ of those diagrams
$C=\left(C_i\dashrightarrow C_j\right)$
such that $C_i$ is cofibrant in $\Cc_i$ for all $i\in I$.
We next show that homotopy is transitive for those ho-morphisms whose source is in $\GCc_{cof}$ and define a composition law between homotopy classes.

\begin{prop}\label{equivalenceia_hoho}Let $A$ be an object of $\GCc_{cof}$. For every object $B$ of $\GCc$,
the homotopy relation is an equivalence relation
on the set of ho-morphisms from $A$ to $B$.
\end{prop}
\begin{proof}
Reflexivity and symmetry are trivial. We prove transitivity. Assume given ho-morphisms $f,f',f'':A\hto B$
together with homotopies $h:f\simeq f'$ and $h':f'\simeq f''$.
Since $A_i$ is cofibrant, for all $i\in I$,
a homotopy $h_i\wt{+}h_i':f_i\simeq f_i'$ is given as in the proof or Proposition
$\ref{relequiv}$ by $h_i\wt{+}h_i':=\nabla_{B_i}\Ll_i$, where $\Ll_i:A_i\to P^2(B_i)$ 
satisfies $\pi_{B_i}\Ll_i=(h_i,h_i')$ and $\pi_{B_i}=(\delta_{P(B_i)}^0,P(\delta_{B_i}^1))$.
Consider the commutative solid diagram:
$$
\xymatrix{
&A_i \ar@/_/[ddr]_{(H_u,H_u')} \ar@/^/[drrr]^{(\Ll_j\varphi_u,\varphi_u\Ll_i)} \ar@{.>}[dr]|-{\Ll_u}\\
&& P^3(B_j) \ar[d]^{P(\pi_{B_j})} \ar[rr]^-{} && P^2(B_j)\times P^2(B_j) \ar[d]^{\pi_{B_j}\times\pi_{B_j}}\\
&& P(\Pp(\delta^1_{B_j})) \ar[rr]^-{} && \Pp(\delta^1_{B_j})\times \Pp(\delta^1_{B_j})&.
}
$$
Since $\pi_{B_j}$ is a trivial fibration, by Lemma $\ref{holift}$
there exists a dotted arrow $\Ll_u$, making the diagram commute. Let $H_u\wt{+}H_u':=P(\nabla_{B_j})\Ll_u$.
Then the family $h\wt{+}h':=(h_i\wt{+}h_i',H_u\wt{+}H_u')$ is a homotopy of ho-morphisms from $f$ to $f''$.
\end{proof}

We will denote by
$[A,B]^h:=\GCch(A,B)/\sim$
the set of ho-morphisms from $A$ to $B$ modulo the equivalence 
relation transitively generated by the homotopy relation.

\begin{lem}
\label{compohomos}Let $A$ be an object of $\GCc_{cof}$. For every pair of objects $B, C$ of $\GCc$,
there is a map
$$
*:[A,B]^h\times [B,C]^h\lra [A,C]^h
$$
such that:
\begin{enumerate}[(1)]
\item  If either $g$ or $f$ are morphisms of $\GCc$, then $[g]*[f]=[gf]$, where $gf$ is the composition
defined in Lemma $\ref{compo_senzilles}$.
\item If $h$ is a morphism and $f,g$ are ho-morphisms, then
$$[h]*([g]*[f])=[hg]*[f].$$
\end{enumerate}
\end{lem}
\begin{proof}
Given $[f]\in [A,B]^h$ and $[g]\in [B,C]^h$, choose representatives $f=(f_i,F_u)$ and $g=(g_i,G_u)$ respectively. 
We then let
$$g*f:=(g_if_i,P(g_j)F_u\wt{+} G_uf_i),$$ where $P(g_j)F_u\wt{+} G_uf_i=\nabla_{C_j}\Ll_u$ is defined as in the proof of Proposition $\ref{relequiv}$.
Let $g':B\hto C$ such that $h:g\simeq g'$.
Then $g'*f=(g_i'f_i,P(g_j')F_u\wt{+} G_u'f_i)$ where $P(g_j')F_u\wt{+} G_u'f_i=\nabla_{C_j}\Ll_u'$.
We next show that $g*f\simeq g'*f$. 

Let $\Gamma_u=(P(h_j)F_u,H_uf_i)$ and consider the commutative solid diagram:
$$
\xymatrix{
&A_i \ar@/_/[ddr]_{\Gamma_u} \ar@/^/[drrr]^{(\Ll_u,\Ll_u')} \ar@{.>}[dr]|-{K_u}\\
&& P^3(C_j) \ar[d]^{P(\pi_{C_j})} \ar[rr]_{} && P^2(C_j)\times P^2(C_j) \ar[d]^{\pi_{C_j}\times\pi_{C_j}}\\
&& P(\Pp(\delta^1_{C_j})) \ar[rr]_{}&& \Pp(\delta^1_{C_j})\times \Pp(\delta^1_{C_j})&.
}
$$
Since $\pi_{C_j}$ is a trivial fibration, by Lemma $\ref{holift}$
there exists a dotted arrow $K_u$, making the diagram commute. Let 
$$H_u':=\mu_{C_j}P(\nabla_{C_j})K_u:A_i\to P^2(C_j).$$
Then $(P(g_i)h_i,H_u')$ is a homotopy from $g*f$ to $g*f'$.

Analogously, if $f\simeq f'$ one proves that $g*f\simeq g*f'$. 
By Proposition $\ref{equivalenceia_hoho}$ the homotopy relation is transitive between ho-morphisms 
for which the source is in $\GCc_{cof}$. Therefore the class $[g*f]$ does not depend on the chosen representatives and liftings,
and the map $[g]*[f]:=[g*f]$ is well defined.

Let us prove (1). Let $[f]\in [A,B]^h$, and let $g:B\to C$ be a morphism. Choose a representative $f$ of $[f]$, and let $gf=(g_if_i,P(g_i)F_u)$. By Lemma
$\ref{compo_senzilles}$ this is a well defined ho-morphism from $A$ to $C$. We next show that
$[g]*[f]=[gf]$, when $g$ is considered as a ho-morphism with $G_u=(\iota_{C_j}\varphi_ug_i)$.
Consider the diagram
$$
\xymatrix{
&P^2(C_j)\ar@{->>}[d]^{\pi_{C_j}}\\
A_i\ar[ur]^{\Ll_u}\ar[r]^{\gamma_u}&\Pp(\delta^1_{C_j}),
}
$$
where $\gamma_u=(P(g_j)F_u, \iota_{C_j} g_j\varphi_uf_i)$, and 
$\Ll_u:=\iota_{P(C_j)}P(g_j)F_u$. By the naturality of $\delta^k$ and $\iota$ it follows that
$$
\left\{\begin{array}{l}
\delta^0_{P(C_j)}\Ll_u=\delta^0_{P(C_j)}\iota_{P(C_j)}P(g_j)F_u=P(g_j)F_u.\\
     P(\delta^1_{C_j})\Ll_u=\iota_{C_j}\delta^1_{C_j}P(g_j)F_u=\iota_{C_j}g_j\delta^1_{B_j}F_u=\iota_{c_j}g_j\varphi_uf_i.   
       \end{array}\right.
$$
Therefore the above diagram commutes. By definition,
 the folding map $\nabla$ (see Definition $\ref{foldingmap}$) satisfies $\nabla_{C_j}\iota_{P(C_j)}=1_{P(C_j)}$.
It follows that
$$K_u:=\nabla_{P(C_j)}\Ll_u=P(g_j)F_u.$$
Therefore $[g]*[f]=[gf]$.
The proof for the other composition follows analogously.\\

Let us prove (2). Let $f:A\hto B$ and $g:B\hto C$ be ho-morphisms, and let
$\gamma_u=(P(g_j)F_u,G_uf_i)$. We have
$f*g=(f_ig_i,K_u)$, where $K_u=\nabla_{P(C_j)}\Ll_u$, and $\Ll_u$ is an arbitrary morphism 
satisfying $\pi_{C_j}\Ll_u=\gamma_u$. If $h:C\to D$ is a morphism, by $(1)$ we have
$$[h]*([g]*[f])=[(h_ig_if_i,P(h_j)K_u)].$$
On the other hand, let $\gamma_u'=(P(h_jg_j)F_u,P(h_j)G_uf_i)$, and define a morphism
$\Ll_u':=P^2(h_j)\Ll_u$. Then $\pi_{D_j}\Ll_u'=\gamma_u'$. Therefore
$hg*f=(h_jg_jf_j,K_u')$, where $K_u'=\nabla_{D_j}\Ll_u'=P(h_j)K_u$.
The identity $[hg]*[f]=[h]*([g]*[f])$ follows.
\end{proof}

\subsection{Localization with respect to ho-equivalences}
We next show that if $A\in\GCc_{cof}$, the factorization map
$\Phi_{A,B}:\GCch(A,B)\to\GCc[\Hh^{-1}](A,B)$
is a bijection. This allows to define a category $\pGCch_{cof}$ whose objects are those of $\GCc_{cof}$
 and whose morphisms are homotopy classes of ho-morphisms. We then show that this is 
equivalent to the relative localization $\GCc_{cof}[\Hh^{-1},\GCc]$, the full subcategory of 
$\GCc[\Hh^{-1}]$ whose objects are in $\GCc_{cof}$.
\begin{nada}\label{defini_psi}
For $A\in \GCc_{cof}$ and $B\in\GCc$ define a map
$\Psi_{A,B}:\GCc[\Hh^{-1}](A,B)\lra [A,B]^h$ as follows.

A morphism $f$ of $\GCc[\Hh^{-1}]$ can be represented by a zigzag 
where the arrows going to the left are morphisms of $\Hh$. We define $\Psi$ inductively over the length of the zigzag.
Let $\Psi(1_A)=[1_A]$ and assume that $\Psi(f')$ is defined for a zigzag $f'\in\GCc[\Hh^{-1}](A,B)$. We consider two cases:
\begin{enumerate}[(1)]
 \item If $f=gf'$, where $g:B\to C$ is a morphism of $\GCc$ then
$\Psi(f):=[g]*\Psi(f').$
\item If $f=g^{-1}f'$ where $g:C\to B$ is a ho-equivalence, then
$\Psi(f):=[h]*\Psi(f'),$
where $h:B\hto C$ is a homotopy inverse of $g$.
If $h'$ is another homotopy inverse of $g$, then $h'\sim h'gh\sim h$, and so $[h]=[h']$.
Hence $\Psi$
does not depend on the chosen homotopy inverse.
\end{enumerate}
\end{nada}
\begin{lem}
Let $A$ be an object of $\GCc_{cof}$. The map
$\Psi_{A,B}:\GCc[\Hh^{-1}](A,B)\to [A,B]^h$
induced by $f\mapsto \Psi(f)$ is well defined
for any object $B$ of $\GCc$.
\end{lem}
\begin{proof}
We need to prove that the definition does not depend on the chosen representative, that is,
given a hammock between zig-zags $f$ and $\wt f$, then $\Psi(f)=\Psi(\wt f)$.
The proof is based on the fact that, given the commutative diagram
on the left,
$$
\xymatrix{\ar @{} [dr] |{=}
D\ar[d]_{\alpha}&C\ar[l]_{g}\ar[d]^{\beta \,\,\,\,\,\,\,\,\,\,\Longrightarrow}\\
\wt D&\wt C\ar[l]^{\wt g}
}\,\,\,\,\,\,
\xymatrix{\ar @{} [dr] |{\simeq}
D\ar@{~>}[r]^{h}\ar[d]_{\alpha}&C\ar[d]^{\beta}\\
\wt D\ar@{~>}[r]_{\wt h}&\wt C&,
}
$$
where $g$ and $\wt g$ is are ho-equivalences,
then the diagram on the right commutes up to  homotopy,
where $h$ and $\wt h$ are homotopy inverses of $g$ and $\wt g$ respectively.

It suffices to consider the case when
$f$ and $\wt f$ are related by a hammock of height 1. 
Let
$$
\xymatrix{
f:=\\
\wt f:=
}
\xymatrix{
\ar@{=}[d]A\ar[r]^{f_1}&\ar[d]^{\alpha_1}D_1&\ar[d]^{\beta_1}C_1\ar[l]_{g_1}\ar[r]^{f_2}&\ar[d]^{\alpha_2}D_2&\ar[d]^{\beta_2}C_2\ar[l]_{g_2}\ar[r]^{f_3}&\cdots\ar[r]&\ar[d]^{\alpha_r}D_r&B\ar[l]_{g_r}\ar@{=}[d]\\
A\ar[r]^{\wt f_1}&\wt D_1&\wt C_1\ar[l]_{\wt g_1}\ar[r]^{\wt f_2}&D_2&\wt C_2\ar[l]_{\wt g_2}\ar[r]^{\wt f_3}&\cdots\ar[r]&D_r&B\ar[l]_{\wt g_r}
}
$$
be a commutative diagram,
where $g_k:C_k\to D_k$ and $\wt g_k: \wt C_k\to \wt D_k$ are compositions of ho-equivalences.
Let $f(0)=\wt f(0)=1_A$ and for all $0<k\leq r$, define 
$$f(k):=g_k^{-1}f_k\cdots g_1^{-1}f_1,\text{ and } \wt f(k):=\wt g_k^{-1}\wt f_k\cdots \wt g_1^{-1}\wt f_1.$$
Write $g_k=g_k^1\cdots g_k^{n_k}$ and
$\wt g_k=\wt g_k^1\cdots \wt g_k^{\wt n_k}$, and let
$h_k^j$ and $\wt h_k^j$ be homotopy inverses of $g_k^j$ and $\wt g_k^j$ respectively. With these notations we have
\setlength{\extrarowheight}{0.2cm}
$$\begin{array}{l}\Psi(f(k)):=[h_{k}^{n_k}]*(\cdots *([h_k^2]*([h_k^1]*([f_k]*\Psi(f(k-1)))))),\\
\Psi(\wt f(k)):=[\wt h_{k}^{n_k}]*(\cdots *([\wt h_k^2]*([\wt h_k^1]*([\wt f_k]*\Psi(\wt f(k-1)))))).\end{array}$$
\setlength{\extrarowheight}{0.0cm}

From the definition it follows that:
\begin{equation}
[g_k]*\Psi(f(k))=[f_k]* \Psi(f(k-1)).
\tag{{$p_k$}}\end{equation}
We will now proceed by induction.
Assume that for all $n<k$ we have
\begin{equation}
\Psi(\wt f(n))=[\beta_n]*\Psi(f(n)).
\tag{{$h_n$}}\end{equation}
For the following identities we will constantly use $(1)$ and $(2)$ of Lemma $\ref{compohomos}$.
We have:
\setlength{\extrarowheight}{0.2cm}
$$
\begin{array}{l@{\,}l@{\,}l@{\,\,\,\,\,}l}
\Psi(\wt f(k))&=& [\wt h_k^{n_k}]*(\cdots *([\wt h_k^1]*([\wt f_k]*\Psi(\wt f(k\text{-}1)))))&\text{( by ($h_{k-1}$) )}\\
&=& [\wt h_k^{n_k}]*(\cdots *([\wt h_k^1]*([\wt f_k\beta_{k-1}]*\Psi(f(k\text{-}1)))))&\text{( $\wt f_k\beta_{k-1}=\alpha_kf_k$ )}\\
&=& [\wt h_k^{n_k}]*(\cdots *([\wt h_k^1]*([\alpha_kf_k]*\Psi(f(k\text{-}1)))))&\text{( by ($p_{k}$) )}\\
&=& [\wt h_k^{n_k}]*(\cdots *([\wt h_k^1]*([\alpha_kg_k]*\Psi(f(k))&\text{( $\alpha_kg_k=\wt g_k\beta_k$ )}\\
&=& [\wt h_k^{n_k}]*(\cdots *([\wt h_k^1]*([\wt g_k\beta_k]*\Psi(f(k))&\text{( $\wt g_k=\wt g_k^1\cdots \wt g_k^{\wt n_k}$ )}\\
&=& [\beta_k]*\Psi(f(k)).
\end{array}
$$
\setlength{\extrarowheight}{0.0cm}

Since $\beta_r=1_B$ we get
$\Psi(\wt f)=\Psi(\wt f(r))=\Psi(f(r))=\Psi(f).$
\end{proof}

\begin{prop}\label{bijeccio_homorfs}
Let $A$ be an object of $\GCc_{cof}$. The maps
$$
\Phi_{A,B}:[A,B]^h\rightleftarrows\GCc[\Hh^{-1}](A,B):\Psi_{A,B}$$
are inverses to each other, for every object $B$ of $\GCc$.
\end{prop}
\begin{proof}To simplify notation, we omit the subscripts of both $\Psi$ and $\Phi$.
Let $[f]$ be an element of $[A,B]^h$. Then 
$$\Psi(\Phi([f]))=\Psi(\{q_fp_f^{-1}\})=[q_f]*[\iota_f]=[q_f\iota_f]=[f].$$
For the other composition,
we proceed by induction as follows. Assume that for $f'\in \GCc[\Hh^{-1}](A,B)$
we have $\Phi(\Psi(f'))=f'$. We consider two cases:
\begin{enumerate}[(1)]
 \item If $f=gf'$, where $g:B\to C$ is a morphism of $\GCc$ then
$$\Phi(\Psi(f))=\Phi(\Psi(gf'))=\Phi([g]* \Psi(f))=\Phi([g])\circ f=g\circ f.$$
\item If $f=g^{-1}f'$ where $g:C\to B$ is a ho-equivalence, then
$\Psi(g^{-1}f)=[h]* \Psi(f)$, where $h$ is a homotopy inverse of $g$.
By Lemma $\ref{p2}$ we can write $\{g\}=\Phi([g])$. Therefore
$$g\circ \Phi(\Psi(g^{-1}f))=\Phi([g])\circ \Phi([h]*\Psi(f))=\Psi(f).$$
If we compose on the left by $g^{-1}$ we obtain
$\Phi(\Psi(g^{-1}f))=g^{-1}\circ f.$
\end{enumerate}
\end{proof}

\begin{lem}\label{transfer_additiu}Let $A,B,C$ be objects of $\GCc_{cof}$ and let
 $[f]\in [A,B]^h$ and $[g]\in [B,C]^h$. Then
$$[g]*[f]=\Psi_{A,C}\left(\Phi_{B,C}([g])\circ \Phi_{A,B}([f])\right).$$
\end{lem}
\begin{proof}
Since $A,B,C$ are objects of $\GCc_{cof}$ the maps $\Psi_{A,-}$, $\Psi_{B,-}$ and $\Psi_{C,-}$ are well defined.
For the rest of the proof we omit the subscripts of $\Psi$ and $\Phi$.
By definition we have:
$$\Psi(\Phi([g])\circ \Phi([f]))=\Psi(\{q_gp_g^{-1}\}\circ \{q_fp_f^{-1}\})=
[q_g]*([\iota_g]*([q_f]*[\iota_f])).$$
Since $q_f$ and $q_g$ are morphisms of $\GCc$, we have $[q_g]*[\iota_g]=[q_g\iota_g]=[g]$, and 
$[q_f]*[\iota_f]=[q_f\iota_f]=[f]$. The result follows from (2) of Lemma $\ref{compohomos}$.
\end{proof}

\begin{teo}\label{equiv_homorf_hoequiv}
The objects of $\GCc_{cof}$ with the homotopy classes of ho-morphisms define a category
$\pGCch_{cof}$. There is an equivalence of categories
$$\Phi:\pGCch_{cof} {\rightleftarrows} \GCc_{cof}[\Hh^{-1},\GCc]:\Psi$$
\end{teo}
\begin{proof}
Let $f:A\hto B$ and $g:B\hto C$ be two ho-morphisms between objects of $\GCc_{cof}$,
We first show that
$[g]*[f]=\Psi\left(\Phi([g])\circ \Phi([f])\right).$
By definition we have:
$$\Psi(\Phi([g])\circ \Phi([f]))=\Psi(q_gp_g^{-1}\circ q_fp_f^{-1})=
[q_g]*([\iota_g]*([q_f]*[\iota_f])).$$
Since $q_f$ and $q_g$ are morphisms of $\GCc$, we have $[q_g]*[\iota_g]=[q_g\iota_g]=[g]$, and 
$[q_f]*[\iota_f]=[q_f\iota_f]=[f]$. The result follows from (2) of Lemma $\ref{compohomos}$.

If $h:C\hto D$ is another ho-morphism between objects of $\GCc_{cof}$, we have:
$$[h]*([g]*[f])=
\Psi\left(\Phi([h])\circ\Phi(\Psi(\Phi([g])\circ \Phi([f])))\right).
$$
By Proposition $\ref{bijeccio_homorfs}$ we have $\Phi\Psi=1$, and hence,
$$[h]*([g]*[f])=
\Psi\left(\Phi([h])\circ\Phi([g])\circ \Phi([f]))\right)=([h]*[g])*[f].
$$
Therefore the composition of $\pGCch_{cof}$ is associative.
The equivalence of categories follows from Proposition $\ref{bijeccio_homorfs}$.
\end{proof}

\subsection{A Cartan-Eilenberg structure}
We next show that the triple $(\GCc,\Hh,\Ww)$ is a left Cartan-Eilenberg category, 
where the category $\GCc_{cof}$ of level-wise cofibrant objects 
is a full subcategory of cofibrant models.
\begin{lem}\label{trivfib_levanta}
Let $C$ be an object of $\GCc_{cof}$. For every diagram 
$$
\xymatrix{
&A\ar@{->>}[d]^w_{\wr}\\
C\ar@{.>}[ur]\ar@{~>}[r]^f&B&,
}
$$
where $w$ is a trivial fibration of $\GCc$ and $f$ is a ho-morphism,
there is a ho-morphism $g:C\hto A$ making the diagram commute.
\end{lem}
\begin{proof}
Since $C_i$ is cofibrant for each $i\in I$, there are maps $g_i:C_i\to A_i$ such that $w_ig_i=f_i$.
We have
$w_jg_j\varphi_u=f_j\varphi_u{\simeq}\varphi_uf_i=\varphi_uw_ig_i=w_j\varphi_ug_i.$
Consider the commutative solid diagram
$$
\xymatrix{
C_i \ar@/_1pc/[ddr]_{F_u} \ar@/^1pc/[drrr]^{(g_j\varphi_u,\varphi_ug_i)} \ar@{.>}[dr]|-{G_u}\\
& P(A_j) \ar[d]^{P(w_j)} \ar[rr]^{(\delta^0_{A_j},\delta^1_{A_j})} && A_j\times A_j \ar[d]^{w_j\times w_j}\\
& P(B_j) \ar[rr]_{(\delta^0_{B_j},\delta^1_{B_j})} && B_j\times B_j&.
}
$$
Since $w_j$ is a trivial fibration, by Lemma $\ref{holift}$ there exists
a dotted arrow $G_u$, making the diagram commute.
The family $g=(g_i,G_u)$ is a ho-morphism, and $wg=f$.
\end{proof}

\begin{prop}\label{levanta_homorfismos}
Let $C$ be an object of $\GCc_{cof}$ and let $w:A\to B$ be a weak equivalence in $\GCc$.
The map
$w_*:[C,A]^h\lra [C,B]^h$
defined by $[f]\mapsto [wf]$ is a bijection.
\end{prop}
\begin{proof}
We first prove surjectivity.
Let $f:C\hto B$ be a ho-morphism representing $[f]\in [C,B]^h$.
By Corollary $\ref{factoritza_brown}$ the map $w$ factors as 
a homotopy equivalence $\iota_w$ followed by a trivial fibration $q_w$, 
 giving rise to a solid diagram 
$$
\xymatrix{
&A\ar@<.6ex>[d]^{\iota_w}\ar@/^2pc/[dd]^{w}\\
&\Pp(w)\ar@<.6ex>[u]^{p_w}\ar@{->>}[d]^{q_w}_{\wr}\\
C\ar@{.>}[ur]^{g'}\ar@{~>}[r]^{f}&B&.
}
$$
By Lemma $\ref{trivfib_levanta}$ there is a ho-morphism $g':C \hto \Pp(w)$ such that $q_wg'=f$.
Let $g:=p_wg'$.
We have
$wg=q_w\iota_wg=q_w\iota_wp_wg'\simeq q_wg'= f.$
Therefore $[wg]=[f]$, and $w_*$ is surjective.

To prove injectivity, let $g,g':C\hto A$ be two ho-morphisms, representing $[g]$ and $[g']$ respectively and let $h:wg\simeq wg'$ be a homotopy.
Let $\Pp(w,w)$ denote the double mapping path of $w$,
defined by the fibre product
$$\xymatrix{
\pb\Pp(w_i,w_i)\ar[d]^{}\ar[r]^{}&P(B_i)\ar[d]^{(\delta^0_{B_i},\delta^1_{B_i})}\\
A_i\times A_i\ar[r]^{w_i\times w_i}&B_i\times B_i&,
}$$
for $i\in I$,
together with the comparison morphism $\psi_u=((\varphi_u\times\varphi_u)\pi_1,P(\varphi_u)\pi_2),$ for $u:i\to j$.

The triple $(g,g',h)$ defines a ho-morphism $\gamma:C\hto \Pp(w,w)$. Indeed,
for all $i\in I$, let $\gamma_i$ be the map defined by the pull-back diagram:
$$
\xymatrix{
C_i \ar@/_1pc/[ddr]_{(g_i,g_i')} \ar@/^1pc/[drr]^{h_i} \ar@{.>}[dr]|-{\gamma_i}\\
& \pb\Pp(w_i,w_i) \ar[d]^{} \ar[r]^{} & P(B_i) \ar[d]^{(\delta^0_{B_i},\delta^1_{B_i})}\\
& A_i\times A_i \ar[r]_{w_i\times w_i} & B_i\times B_i&,
}
$$
and for all $u:i\to j$ let $\Gamma_u$ be the map defined by the pull-back diagram:
$$
\xymatrix{
C_i \ar@/_1pc/[ddr]_{(G_u,G_u')} \ar@/^1pc/[drrr]^{H_u} \ar@{.>}[dr]|-{\Gamma_u}\\
& \pb P(\Pp(w_j,w_j)) \ar[d]^{} \ar[rr]^{} && P^2(B_j) \ar[d]^{(P(\delta^0_{B_j}),P(\delta^1_{B_j}))}\\
& P(A_j)\times P(A_j) \ar[rr]_{P(w_j\times w_j)} && P(B_j)\times P(B_j)&.
}
$$
Then the family $\gamma=(\gamma_i,\Gamma_u)$ is a ho-morphism. Indeed,
\setlength{\extrarowheight}{0.2cm}
$$\begin{array}{l}\delta^0_{\Pp(w_j,w_j)}\Gamma_u=((\delta^0_{A_j}G_u,\delta^0_{A_j}G_u'),\delta^0_{P(B_j)}H_u)=((g_j\varphi_u,g_j'\varphi_u),h_j\varphi_u)=\gamma_j\varphi_u,\\
\delta^1_{\Pp(w_j,w_j)}\Gamma_u=((\delta^1_{A_j}G_u,\delta^1_{A_j}G_u'),\delta^1_{P(B_j)}H_u)=((\varphi_ug_i,\varphi_ug_i'),P(\varphi_u)h_i))=\psi_u\gamma_i.\end{array}$$
\setlength{\extrarowheight}{0.0cm}

Consider the solid diagram
$$\xymatrix{
&P(A)\ar[d]^{\ov{w}}_{\wr}\\
C\ar@{.>}[ur]^{\gamma'}\ar@{~>}[r]^{\gamma}&\Pp(w,w)&.}
$$
By Lemma $\ref{induit_we}$ the map $\ov{w}$ defined level-wise by $\ov{w_i}=((\delta^0_{A_i},\delta^1_{A_i}),P(w_i))$ is a weak equivalence.
Hence $\ov{w}_*$ is surjective, and there exists a dotted arrow $\gamma'$ such that $\ov{w}\gamma'\simeq \gamma$.
It follows that $g\simeq \delta^0_A\gamma'\simeq \delta^1_A\gamma'\simeq g'$. Hence
$[g]=[g']$.
\end{proof}

\begin{cor}\label{cofibrants_diagrames}
Objects of $\GCc_{cof}$ are Cartan-Eilenberg cofibrant in $(\GCc,\Hh,\Ww)$, i.e., 
for every weak equivalence $w:A\to B$, and every object $C\in\GCc_{cof}$ the induced map
$w_*:\GCc[\Hh^{-1}](C,A)\to \GCc[\Hh^{-1}](C,B)$
is a bijection.
\end{cor}

We prove the main result of this section.

\begin{teo}\label{diagramesCE0}
Let $\GCc$ be a diagram category satisfying conditions {\normalfont(I$_1$)-(I$_3$)} of $\ref{indexcat}$.
Assume that for each $i\in I$, the categories $\Cc_i$ have enough cofibrant models.
Then the triple $(\GCc,\Hh,\Ww)$ is a left Cartan-Eilenberg category and $\GCc_{cof}$ is a full subcategory of cofibrant models. 
The inclusion induces an equivalence of categories
$\pGCch_{cof}\stackrel{\sim}{\lra} \Ho(\GCc).$
\end{teo}
\begin{proof}
Given an object $A$ of $\GCc$, let
$f_i:C_i\stackrel{\sim}{\lra} A_i$ be a cofibrant model of $A_i$, for each $i\in I$.
By Proposition $\ref{bijecciocofibrants}$, given the solid diagram in $\Cc_j$
$$\xymatrix{
C_i\ar[d]_{f_i}\ar@{.>}[r]^{\varphi_u'}&C_j\ar[d]^{f_j}\\
A_i\ar[r]^{\varphi_u}&A_j
}
$$
there exists a dotted arrow $\varphi_u'$ and a homotopy $F_u:f_j\varphi'_u\simeq \varphi_uf_i$.
This defines an object $C=(\varphi_u':C_i{\lra}C_j)$ of $\GCc_{cof}$ and a ho-morphism $f=(f_i,F_u):C\hto A$
which is a level-wise weak equivalence. By Corollary $\ref{cofibrants_diagrames}$, $C$ is Cartan-Eilenberg cofibrant in $(\GCc,\Hh,\Ww)$
and the map $\Phi_{C,A}(f):C\to A\in \GCc[\Hh^{-1}]$ 
is an isomorphism in $\GCc[\Ww^{-1}]$.
This proves that the triple $(\GCc,\Hh,\Ww)$ is a left Cartan-Eilenberg category with cofibrant models in $\GCc_{cof}$.
By Theorem $2.3.4$ of \cite{GNPR} the inclusion induces an equivalence of categories 
$\GCc_{cof}[\Hh^{-1},\GCc]\stackrel{\sim}{\lra} \Ho(\GCc)$.
The above equivalence follows from Theorem $\ref{equiv_homorf_hoequiv}$.
\end{proof}

To end this section we consider a generalization of Theorem $\ref{Pcat_es_CE}$, in which the category under study is a full subcategory of a diagram category.
This will be useful for the application to mixed Hodge theory.

\begin{teo}\label{diagramesCE}
Let $\GCc$ be a diagram category satisfying conditions {\normalfont(I$_1$)-(I$_3$)} of $\ref{indexcat}$.
Let $\Dd$ be a full subcategory of $\GCc$ such that:
\begin{enumerate}[(i)]
\item The mapping path of a ho-morphism between objects of $\Dd$ is an object of $\Dd$.
\item There is a full subcategory $\Mm\subset\Dd\cap \GCc_{cof}$ such that
for every object $D$ of $\Dd$ there exists an object $M\in\Mm$ together with a
ho-morphism from $M$ to $D$ which is a level-wise weak equivalence.
\end{enumerate}
Then the triple $(\Dd,\Hh,\Ww)$ is a left Cartan-Eilenberg category and $\Mm$ is a full subcategory of cofibrant models.
The inclusion induces an equivalence of categories
$\pi^h\Mm\stackrel{\sim}{\lra} \Ho(\Dd).$
\end{teo}
\begin{proof}
By (i) and Proposition $\ref{inversaho}$, for every pair of objects $A$ and $B$ of $\Dd$, 
the factorization map of $\ref{def_phi}$
restricts to a well defined map
$\Phi_{A,B}:[A,B]^h\lra \Dd[\Hh^{-1}](A,B),$
which preserves weak equivalences. By Theorem $\ref{equiv_homorf_hoequiv}$
this map induces an equivalence
$\Phi:\pi^h\Mm\rightleftarrows \Mm[\Hh^{-1},\Dd]:\Psi.$
The proof follows analogously to that of Theorem $\ref{diagramesCE0}$.
\end{proof}

\section{Multiplicative Mixed Hodge Theory}
We prove a multiplicative version of Beilinson's Theorem for mixed Hodge complexes.
Together with Navarro's functorial construction of mixed Hodge diagrams, this provides functoriality for the mixed Hodge
structures on the homotopy type of complex algebraic varieties.
As an application, we derive the functor of indecomposables of mixed Hodge diagrams,
and give an alternative proof of the fact that the rational homotopy groups of every simply connected complex algebraic
variety inherit functorial mixed Hodge structures.

\subsection{Mixed Hodge diagrams of algebras}
Let $I=\{0\to 1\leftarrow 2\to\cdots\leftarrow s\}$ be an index category of zig-zag type and
fixed length $s$.
We next define the category of diagrams of filtered dga's of type $I$ over $\QQ$. This is a diagram category
whose vertices are categories of filtered and bifiltered dga's defined over $\QQ$ and $\CC$ respectively.
Additional assumptions on the filtrations will lead to the notion of mixed Hodge diagram.
\begin{defi}
Let $\mathbf{A}:I\to\mathsf{Cat}$ be the functor defined by
$$
\xymatrix{
0\ar@{|->}[d]\ar[r]^u&1\ar@{|->}[d]&\ar[l]\cdots\ar[r]&\ar@{|->}[d]s-1&\ar[l]_vs\ar@{|->}[d]\\
\dga{}{\QQ}\ar[r]^{u_*}&\Fdga{}{\CC}&\ar[l]_-{Id}\cdots\ar[r]^-{Id}&\Fdga{}{\CC}&\ar[l]_{v_*}\FFdga{}{\CC}
}
$$
where $u_*(A_\QQ,W)=(A_\QQ,W)\otimes\CC$ is the extension of scalars functor
and $v_*(A_\CC,W,F)=(A_\CC,W)$ is the forgetful functor for the second filtration.
All intermediate functors are identities.
The \textit{category of diagrams of filtered algebras} is the
category of diagrams $\DFAlg{}$ associated with $\mathbf{A}$. Objects of
$\DFAlg{}$ are denoted by $A=((A_\QQ,W)\stackrel{\varphi}{\dashleftarrow\dashrightarrow}(A_\CC,W,F))$.
\end{defi}

\begin{nada}\label{pcatdiag1}Consider the P-category structures of $\Fdga{}{\QQ}$ and $\Fdga{}{\CC}$ given by
the $1$-path together with the classes of $E_1$-quasi-isomorphisms and $E_1$-surjections (see Proposition $\ref{rpstruc_fdga}$).
Likewise, in $\FFdga{}{\CC}$, consider the P-category structure given by
the $(1,0)$-path together with the classes of $E_{1,0}$-quasi-isomorphisms and $E_{1,0}$-surjections (see Proposition $\ref{rpstruc_ffdga}$).
With these choices, the category $\DFAlg{}$ satisfies conditions (I$_1$)-(I$_3$) of $\ref{indexcat}$.
Hence we have the corresponding notions of \textit{ho-morphism} and \textit{ho-equivalence}, as well as the classes $\Hh$ and $\Ww$ of
ho-equivalences and level-wise weak equivalences of $\DFAlg{}$.
Denote by $\Qq$ the class of level-wise quasi-isomorphisms compatible with filtrations.
Since filtrations are regular and exhaustive, we have $\Ww\subset\Qq$.
\end{nada}

\begin{defi}\label{defMHC}
A \textit{mixed Hodge diagram (of dga's over $\QQ$)} is a diagram of filtered algebras
$$A=\left((A_\QQ,W)\stackrel{\varphi}{\dashleftarrow\dashrightarrow}(A_\CC,W,F)\right)$$
satisfying the following conditions:
\begin{enumerate}
\item [($\text{MH}_0$)] The map $\varphi$ is a string of $E_1^W$-quasi-isomorphisms.
The weight filtration $W$ is regular and exhaustive. The Hodge filtration $F$ is biregular and
$H^*(A_\QQ)$ has finite type.
\item [($\text{MH}_1$)] For all $p\in\ZZ$, the differential of $Gr_p^WA_\CC$ is strictly compatible with the filtration $F$.
\item [($\text{MH}_2$)]  For all $n\geq 0$ and all $p\in\ZZ$, the filtration $F$ induced on $H^n(Gr^W_pA_{\CC})$ defines a pure Hodge structure of
weight $p+n$ on $H^n(Gr^W_pA_\QQ)$.
\end{enumerate}
\end{defi}
Denote by $\MHD$ the category of mixed Hodge diagrams of a fixed type $I$,
omitting the index category in the notation. By forgetting the multiplicative structures we recover the category $\MHC$ of mixed Hodge complexes
(see \cite{DeHIII}, 8.1.5).
Axiom
($\mathrm{MH}_2$) implies that for all $n\geq 0$ the triple $(H^n(A_\QQ),\Dec W,F)$ is a mixed Hodge structure,
where $\Dec W$ denotes the d\'{e}calage of the weight filtration (see Definition 1.3.3 of \cite{DeHII}).

The base field $\QQ$ is considered as a mixed Hodge diagram with trivial filtrations and trivial differential.
We will assume that every mixed Hodge diagram $A$ is \textit{0-connected}: the unit map $\eta:\QQ\to A_\QQ$
induces an isomorphism $\QQ\cong H^0(A_\QQ)$. This is a sufficient condition in order to establish the existence of minimal models à la Sullivan.

By Scholie 8.1.9 of \cite{DeHIII} the spectral
sequences associated with the weight and the Hodge filtrations degenerate at the stages $E_2$ and $E_1$ respectively.
Hence the classes $\Ww$ and $\Qq$ coincide in $\MHD$ and
$\Ho(\MHD):=\MHD[\Qq^{-1}]=\MHD[\Ww^{-1}]$.

\begin{defi}
A \textit{mixed Hodge dga} is a filtered dga $(A,d,W)$ over $\QQ$,
together with a filtration $F$ on $A_\CC:=A\otimes_\QQ \CC$, such that for each $n\geq 0$, the triple
$(A^n,\Dec W,F)$
is a mixed Hodge structure and the differentials $d:A^n\to A^{n+1}$ and products
$A^n\otimes A^m\to A^{n+m}$
are morphisms of mixed Hodge structures.
\end{defi}
The cohomology of every mixed Hodge diagram is a mixed Hodge dga with trivial differential.
Following Sullivan's theory one defines a
notion of \textit{minimal mixed Hodge extension} of an augmented mixed Hodge dga 
(see Definition 3.15 of \cite{CG1}).
A mixed Hodge dga is \textit{minimal} if it is the colimit of a sequence of
minimal mixed Hodge extensions starting from the base field $\QQ$
 with trivial mixed Hodge structure.
Every minimal mixed Hodge dga $M$ is free as a filtered dga and satisfies
 $d(W_pM)\subset W_{p-1}(M^+\cdot M^+)$.
Hence it is a Sullivan minimal dga satisfying $\Dec W_pM^n=W_{p-n}M^n$.
Denote by $\MHSA$ the category of minimal mixed Hodge dga's. 

\begin{prop}\label{cofmins}Objects of $\MHSA$ are Cartan-Eilenberg cofibrant and minimal 
in the triple $(\MHD,\Hh,\Qq)$.
\end{prop}
\begin{proof}
Let $M$ be a minimal mixed Hodge dga. By Corollary $\ref{cofibrants_diagrames}$, for $M$ to be cofibrant it suffices to see
that $(M,W)$ and $(M_\CC,W,F)$ are cofibrant in the categories of filtered and bifiltered dga's respectively.
This follows from Theorem 2.12 of \cite{CG1}, which is an adaptation of the classical proof (see Lemma 12.4 of \cite{FHT}),
to the filtered setting.
To see that objects of $\MHSA$ are minimal it remains to prove that every weak equivalence $w:M\to N$
is an isomorphism. Indeed, since both $M$ and $N$ are Sullivan minimal, $w$ is an isomorphism of dga's
(see Proposition 7.6 of \cite{BG}). Since $M$ and $N$ carry mixed Hodge structures,
 the map $w$ is strictly compatible with $W$ and $F$ (see Theorem 2.3.5 of \cite{DeHII}). Therefore
 $w:(M,W)\to (N,W)$ is an isomorphism of filtered dga's, and $w_\CC:(M_\CC,W,F)\to (N_\CC,W,F)$ is an isomorphism of bifiltered dga's.
\end{proof}

\begin{teo}\label{MHDCE}
The triple $(\MHD,\Hh,\Qq)$ is a left Cartan-Eilenberg category and $\MHSA$ is a full subcategory of cofibrant minimal models.
The inclusion induces an
equivalence of categories
$$\pi^h\MHSA\stackrel{\sim}{\lra} \Ho\left(\MHD\right).$$
\end{teo}
\begin{proof}
We show that the conditions of Theorem $\ref{diagramesCE}$ are satisfied,
with $\GCc=\Gamma\mathbf{A}$, $\Dd=\MHD$ and $\Mm=\MHSA$.
Let us prove (i). By Lemma $\ref{inversaho}$,
given a ho-morphism of mixed Hodge diagrams $f:A\hto B$, the level-wise projection
$p:\Pp^h(f)\to A$ is in $\Ww$. Therefore condition ($\text{MH}_0$)
follows from the two out of three property of $E_1$-quasi-isomorphisms.
Since $Gr_p^WP_1(B_\CC)=Gr_p^WB_\CC\otimes\Lambda(t)\oplus Gr_{p+1}^WB_\CC\otimes\Lambda(t)dt$, it follows that
$$H^n(Gr_p^W\Pp(f_\CC))\cong H^n(Gr_p^WA_\CC)\oplus H^n(Gr_p^WB_\CC)\oplus H^{n-1}(Gr_{p+1}^WB_\CC)$$
for all $n\geq 0$ and all $p\in\ZZ$.
Hence conditions ($\text{MH}_1$) and ($\text{MH}_2$) are satisfied.

By Theorem 3.17 of \cite{CG1}, for every cohomologically connected mixed Hodge diagram $A$ there exists a 
minimal mixed Hodge dga $M$ and a ho-morphism $f:M\hto A$ which is a level-wise quasi-isomorphism.
By Proposition $\ref{cofmins}$, $M$ is cofibrant and minimal. Hence (ii) is satisfied.
\end{proof}

\begin{cor}\label{mhstipohomotopia}
The rational homotopy type of every complex algebraic variety carries functorial mixed Hodge structures:
there is a functor $\Hh dg:\Sch{}{\CC}\to \pi^h\MHSA$ whose composition with the forgetful functor 
$U_\QQ:\pi^h\MHSA\to \Ho(\dga{}{\QQ})$
sends $X$ to a minimal model $M_X$ of its Sullivan-de Rham algebra of forms in the homotopy category. 
\end{cor}
\begin{proof}
Deligne's construction of mixed Hodge structures on the cohomology of smooth algebraic varieties \cite{DeHII} can be restated as having a functor
$\Hh dg:\mathbf{V}^2(\CC)\to \MHC$
sending every smooth compactification $U\subset X$ of algebraic varieties over $\CC$ with $D=X-U$ a normal crossings divisor, to a
mixed Hodge complex, which computes the cohomology of $U$. 

A multiplicative version of Deligne's functor  
with values in the category of mixed Hodge diagrams $\Hh dg:\mathbf{V}^2(\CC)\to \MHD$ is given is 
 Theorem 8.1.5 of (\cite{Na}.

The category of mixed Hodge diagrams admits a cohomological descent structure via the Thom-Whitney simple functor
(see Theorem 4.4 of \cite{CG1}).
Hence the extension criterion of functors of \cite{GN} gives a functor
$\Hh dg:\Sch{}{\CC}\to \Ho\left(\MHD\right)$ whose rational component is the Sullivan-de Rham functor.
The result follows from Theorem $\ref{MHDCE}$.
\end{proof}

\subsection{Homotopy and indecomposables}

Given an augmented dga $\eps:A\to \kk$ over $\kk$, denote by $A^+$ the kernel of $\eps$.
The quotient graded vector space $Q(A)=A^+/(A^+\cdot A^+)$ with the induced differential is called the \textit{complex of indecomposables} of $A$.
If $A$ is a (bi)filtered dga and the augmentation is compatible with filtrations, then $Q(A)$ is also (bi)filtered.
This defines a functor $Q$ sending augmented (bi)filtered dga's to (bi)filtered complexes of vector spaces over $\kk$.

\begin{defi}
An \textit{augmented mixed Hodge diagram} is a mixed Hodge diagram $A$ together with a morphism $\eps:A\to \QQ$.
\end{defi}

Denote by $\MHD_*$ the category of augmented mixed Hodge diagrams whose morphisms are compatible with the augmentations.

\begin{defi}
The \textit{complex of indecomposables} of an augmented mixed Hodge diagram $A$
is the diagram of filtered complexes defined by 
$$Q(A)=\left((Q(A_\QQ),W)\stackrel{Q(\varphi)}{\dashleftarrow\dashrightarrow}(Q(A_\CC),W,F)\right).$$
\end{defi}
This defines a functor 
$Q:\MHD_*\lra \DFCx,$
where $\DFCx$ denotes the category of diagrams of filtered complexes, obtained by forgetting the multiplicative structures of $\DFAlg{}$
(see Definition 4.1 of \cite{CG2}).
Every minimal mixed Hodge dga $M$ admits a canonical augmentation, and for all $n\geq 0$, the triple
$(Q(M_\QQ)^n,\Dec W,F)$ is a mixed Hodge structure. In particular $Q(M)$ is a mixed Hodge complex with trivial differential.
We have a functor $$\Dec_W\circ Q:\MHSA\lra \mathbf{G}^+(\MHS)$$
with values in the category of graded mixed Hodge structures,
where $\Dec_W$ is defined by d\'{e}calage of the weight filtration.

\begin{defi}
Let $n\geq 0$ and let $A$ be a 1-connected mixed Hodge diagram. The \textit{$n$-th homotopy group of $A$} is the mixed Hodge structure
defined by $\pi^n(A)=\Dec_W\circ Q(M)$, where $M$ is a cofibrant minimal model of $A$.
\end{defi}
We next check that this definition is correct, in the sense that it is functorial,
and does not depend on the chosen minimal model.

\begin{rmk}
Since every 0-connected mixed Hodge diagram has a cofibrant minimal model,
its $n$-th homotopy groups can be defined in the same manner.
In doing so, one loses functoriality, since
dga's need not induce the same morphism of indecomposables, unless the homotopy is
augmented.
For 1-connected dga's,
every homotopy of augmented morphisms is augmented. Hence,
the homotopy is independent of the augmentation.
\end{rmk}

\begin{defi}\label{augm_homotopy}
Let $f,g:A\hto B$ be ho-morphisms of mixed Hodge diagrams. A homotopy $h:A\hto P(B)$ from $f$ to $g$ is said to be \textit{augmented}
if the following diagram commutes.
$$\xymatrix{
A\ar[d]_h\ar[r]^\eps&\QQ\ar[d]^\iota\\
P(B)\ar[r]^{P(\eps)}&P(\QQ)
}$$
\end{defi}

\begin{lem}\label{preserva_equivs_htpho}Every augmented homotopy $h:A\hto P(B)$ between ho-morphisms of 1-connected
mixed Hodge diagrams induces a homotopy 
$$\inte h:Q(A)\hto Q(B)[-1].$$
between ho-morphisms of diagrams of filtered complexes.
\end{lem}
\begin{proof}
A homotopy $h:A\hto P(B)$ from $f=(f_i,F_u)$ to $g=(g_i,G_u)$ is given by a family of homotopies $h_i:A_i\to P(B_i)$ from $f_i$ to $g_i$,
together with second homotopies $H_u:A_i\to P^2(B_j)$ satisfying the conditions of Definition $\ref{hohomotopies}$.
Consider the homogeneous linear map 
$$\inte :P(A_i)=A_i\otimes\Lambda(t,dt)\longrightarrow A_i$$
of degree $-1$ defined by
$a\otimes t^i\mapsto 0$ and $a\otimes t^i dt\mapsto (-1)^{|a|} {{a}\over {i+1}}$ (see \cite{GM}, X.10.3).
Then
$$d\inte h_i+\inte dh_i=g_i -f_i.$$
Therefore the map $\int^0_1h_i:A_i\to B_i[-1]$ is a homotopy of complexes. Likewise, we find that
$$\inte \inte H_ud-d\inte \inte H_u=\inte G_u-\inte F_u+\inte h_j\varphi_u-\varphi_u\inte h_i.$$
Hence the family of pairs $$\inte h:=(\inte h_i,\inte\inte H_u)$$
is a homotopy of ho-morphisms of mixed Hodge complexes from $f$ to $g$ (see Definition 3.8 of \cite{CG2}).
Since $h$ is augmented, this homotopy satisfies
$$\inte h (A^+)\subset B^+\text{ and }\inte h(A^+\cdot A^+)\subset B^+\cdot B^+.$$
Therefore it induces a homotopy at the level of complexes of indecomposables.
\end{proof}

\begin{teo}\label{leftderivedMHD}
The functor of indecomposables admits a left derived functor
$$\mathbb{L}Q:\Ho\left(\MHD^1_*\right)\lra \Ho\left(\MHC\right).$$
The composition of functors
$$\Ho\left(\MHD^1\right)\stackrel{\sim}{\longleftarrow}\Ho\left(\MHD^1_*\right)\xra{\mathbb{L}Q}\Ho\left(\MHC\right)
\xra{\Dec_W\circ H}\mathbf{G}^+(\MHS)$$
defines a functor
$$\pi^*:\Ho\left(\MHD^1\right)\lra\mathbf{G}^+(\MHS)$$
which associates to every 1-connected mixed Hodge diagram $A$, a graded mixed Hodge structure $\pi^*(A)=(Q(M_\QQ),\Dec W,F)$,
where $M$ is a cofibrant minimal model of $A$.
\end{teo}
\begin{proof}
By Theorem $\ref{MHDCE}$ the triple $(\MHD,\Hh,\Qq)$ is a Cartan-Eilenberg category.
By Lemma $\ref{preserva_equivs_htpho}$ the functor $Q:\MHD_*\lra \DFCx,$ sends morphisms in $\Hh$ to level-wise weak equivalences.
By Lemma 3.1.3 of \cite{GNPR}, it admits a left derived functor
$\mathbb{L}Q:\Ho\left(\MHD^1_*\right)\to \DFCx[\Ww^{-1}]$ defined by $\mathbb{L}Q(A)=Q(M_A)$,
where $M$ is a cofibrant minimal model of $A$.
Since $Q(M)$ is a mixed Hodge complex,
$\mathbb{L}Q$ takes values in $\Ho\left(\MHC\right).$
It remains to show that the forgetful functor 
$\Ho\left(\MHD^1_*\right)\to \Ho\left(\MHD^1\right)$ is an equivalence of categories. 
This follows as in the case of 1-connected dga's, from the fact that every 
homotopy between morphisms of 1-connected minimal dga's is augmented
(see Lemma 12.5 of \cite{GM}).
\end{proof}

\begin{cor}[\cite{Mo}, Thm. 9.1, \cite{Na}, Thm. 9.3.2]\label{mhd_mhs}
The rational homotopy groups of every simply connected complex algebraic variety carry functorial mixed Hodge structures.
\end{cor}
\begin{proof}
As in the proof of Corollary $\ref{mhstipohomotopia}$ 
we have a functor
$\Hh dg:\Sch{}{\CC}\to \Ho\left(\MHD\right)$ whose rational component is the Sullivan-de Rham functor.
By Theorem $\ref{leftderivedMHD}$ this gives a functor
$\pi_*:=\pi^*\circ \Hh dg:\Sch{1}{\CC}\to\mathbf{G}^+(\MHS)$ with values in the category of graded mixed Hodge structures.
The rational components of $\pi_*(X)$ are the rational homotopy groups of $X$.
\end{proof}

\section*{Acknowledgments}
I would like to thank F. Guill\'{e}n
for invaluable support and helpful suggestions
and V. Navarro for sharing his
insightful ideas.

\bibliographystyle{amsalpha}
\bibliography{../bibliografia}

\end{document}